\documentclass[11pt]{amsart}%
\usepackage{amssymb,amsfonts,latexsym}
\usepackage{graphics,verbatim}
\usepackage{graphicx}
\usepackage[usenames]{color}
\usepackage{amsmath}
\usepackage{amsfonts}
\usepackage{amssymb}
\usepackage{todonotes}%
\providecommand{\U}[1]{\protect\rule{.1in}{.1in}}
\newtheorem{theorem}{Theorem}[section]
\newtheorem{lem}[theorem]{Lemma}
\newtheorem{lemma}[theorem]{Lemma}
\newtheorem{proposition}[theorem]{Proposition}

\theoremstyle{definition}
\newtheorem{definition}[theorem]{Definition}
\newtheorem{example}[theorem]{Example}

\theoremstyle{remark}
\newtheorem{remark}[theorem]{Remark}
\newtheorem{ques}{Question}

\numberwithin{equation}{section}

\newcommand{\R}{{\mathbb R}}

\begin{document}
\title{On exponential bases and frames with non-linear phase functions and some applications}
\author{Jean-Pierre Gabardo}
\author{Chun-Kit Lai}
\author{Vignon Oussa}
\address{[Jean-Pierre Gabardo] Department of Mathematics, McMaster University, Hamilton, Canada.}
\email{gabardo@mcmaster.ca}

\address{[Chun-Kit Lai] Department of Mathematics, San Francisco State University, San Francisco, US.}
\email{cklai@sfsu.edu}
\address{[Vignon Oussa] Department of Mathematics, Bridgewater State University
Massachusetts, US.}
\email{vignon.oussa@bridgew.edu}
\maketitle

\begin{abstract}
In this paper, we study the spectrality and frame-spectrality of exponential
systems of the type $E(\Lambda,\varphi) = \{e^{2\pi i \lambda\cdot\varphi(x)}:
\lambda\in\Lambda\}$ where the phase function $\varphi$ is a Borel measurable
which is not necessarily linear. A complete characterization of pairs
$(\Lambda,\varphi)$ for which $E(\Lambda,\varphi)$ is an orthogonal basis or a
frame for $L^{2}(\mu)$ is obtained. In particular, we show that the
middle-third Cantor measures and the unit disc, each admits an orthogonal
basis with a certain non-linear phase. Under a natural regularity condition on
the phase functions, when $\mu$ is the Lebesgue measure on $[0,1]$ and
$\Lambda= {\mathbb{Z}},$ we show that only the standard phase functions
$\varphi(x) = \pm x$ are the only possible functions that give rise to
orthonormal bases. Surprisingly, however we prove that there exist a greater
degree of flexibility, even for continuously differentiable phase functions in
higher dimensions. For instance, we were able to describe a large class of
functions $\varphi$ defined on ${\mathbb{R}}^{d}$ such that the system
$E(\Lambda,\varphi)$ is an orthonormal basis for $L^{2}[0,1]^{d}$ when
$d\geq2.$ Moreover, we discuss how our results apply to the discretization
problem of unitary representations of locally compact groups for the
construction of orthonormal bases. Finally, we conclude the paper by stating
several open problems.

\end{abstract}

%\date{29 November 2011}

\section{Introduction}

\subsection{Definition and Background.}

Let $\Lambda$ be a countable set in ${\mathbb{R}}^{d}$. We denote by
$E(\Lambda): = \{e^{2\pi i \lambda\cdot x}:\lambda\in\Lambda\}$ the collection
of exponential functions with frequencies in $\Lambda.$ For a fixed positive
Borel measure $\mu$ defined on a subset $\Omega$ of $\mathbb{R}^{d},$ perhaps
one of the most profound and largely unresolved questions in harmonic analysis
is to characterize pairs of the type $\left( \Omega,\Lambda\right) $ such that
$E(\Lambda)$ is either an orthogonal basis or a Riesz basis or a frame for
$L^{2}\left( \Omega,d\mu\right) .$ Although the literature contains several
significant results, which we summarized below, at this point, this problem
remains unresolved even in a variety of concrete cases.

\medskip

First, we recall that a finite Borel measure is a \textit{spectral measure} if
there exists a countable discrete set $\Lambda$ called the \textit{spectrum},
such that $E(\Lambda)$ forms an orthogonal basis for $L^{2}(\mu)$.
Additionally, if there exists a countable set $\Lambda$ such that $E(\Lambda)$
forms an orthogonal basis for $L^{2}(\Omega)$, we say that $\Omega$ is a
\textit{spectral set} (equivalently, the measure $\chi_{\Omega}\,dx$ is a
spectral measure). The problem of characterizing spectral sets was initiated
by Fuglede \cite{Fug74}, who conjectured in 1974 that a set is spectral if and
only if one can tile the whole Euclidean space $\mathbb{R}^{d}$ using
translates of that set. Although Fuglede's conjecture in its full generality
was disproved by Tao \cite{Tao2004} in 2004, the problem remains open in many
special settings. As a partial solution to this conjecture, Lev and Matolcsi
recently settled affirmatively the case where $\Omega$ is a convex set
\cite{LM2019}. Moreover, Jorgensen and Pedersen \cite{JP1998} were the first
to discover spectral measures that are singularly continuous with respect to the Lebesgue meaure. They found that the
middle-fourth Cantor measure is a spectral measure, while the middle-third
Cantor measure is not. The reader interested in learning more about fractal
spectral measures can consult the following survey \cite{DLW2017} authored by
Dutkay, Lai, and Wang.

%We summarize below partial solutions to this problem.
\medskip

\medskip

Next, let us mention the notion of frame, which generalizes the notion of
basis, in a separable Hilbert space $\mathcal{H}$ and were first introduced by
Duffin and Schaeffer \cite{DS1952} in order to deal with nonharmonic Fourier
series. Given a countable index set $\mathcal{N}$, a collection of vectors
$\{x_{n}\}_{n\in\mathcal{N} }$ in $\mathcal{H}$ is called a \textit{frame} if
there exist positive constants $A,B$ such that
\[
A\, \|x\|^{2}\le\sum_{n\in\mathcal{N}}\, |\langle x, x_{n}\rangle|^{2}\le
B\,\|x\|^{2},\quad x\in\mathcal{H}.
\]
In that case, every element $x$ of $\mathcal{H}$ admits an unconditional
expansion $x=\sum_{n\in\mathcal{N}} \,c_{n}\,x_{n}$, in terms of the frame
elements $x_{n}$, for certain coefficients $c_{n}$.  In the particular case
where $\mathcal{H}=L^{2}(\mu)$ and $\{x_{n}\}_{n\in\mathcal{N} }=E(\Lambda)$
(where each exponential in $E(\Lambda)$ is viewed as an element of $L^{2}%
(\mu)$ in the obvious way), we call $E(\Lambda)$ a \textit{Fourier frame} for
$L^{2}(\mu)$ if the inequalities above hold. Thus, more explicitly, this means
that there exist constants $0<A\le B<\infty$ such that
\begin{equation}
\label{Eq_Frame}A \|f\|_{L^{2}(\mu)}^{2} \le\sum_{\lambda\in\Lambda} \left|
\int f(x)e^{-2\pi i \lambda\cdot x} d\mu(x)\right|  ^{2} \le B\|f\|_{L^{2}%
(\mu)}^{2}, \ \forall f\in L^{2}(\mu).
\end{equation}
When such a frame exists, we call $\mu$ a \textit{frame-spectral measure} and
$\Lambda$ a \textit{frame-spectrum}. Fourier frames provide thus, for each
$f\in L^{2}(\mu)$, a basis-like expansion of the type
\[
f = \sum_{\lambda\in\Lambda} c_{\lambda} e^{2\pi i \lambda\cdot x}%
\]
for some (possibly non-unique) square-summable sequence $(c_{\lambda
})_{\lambda\in\Lambda}.$ The redundancy property built into the structure of
frames makes the expansion of vectors described above robust to the loss of
data, and as such, frames are very useful in signal transmission. In cases
where the coefficients used to represent vectors are unique, the system
$E(\Lambda)$ is called a \textit{Riesz basis} and $\mu$ a
\textit{Riesz-spectral measures}. We refer readers interested in a detailed
treatment of the theory of frames and Riesz bases to the monographs of Heil
and Christensen \cite{Ole,heil2010basis}.

Regarding the existence of exponential frames and Riesz bases, there is a
general interest in determining measures that are Riesz-spectral and
frame-spectral. A theory addressing these issues was established in
\cite{HLL}, and the existence of exponential Riesz basis on a finite union of
rectangles on ${\mathbb{R}}^{d}$ was recently confirmed by Kozma and Nitzan
\cite{KN2015, KN2016}. It is worth noting that despite the intense efforts
devoted to this line of research, to this date, several concrete cases are
still far from being settled. For instance, the following questions have yet
to be resolved (see \cite{Str2000,DLW2017})

\begin{enumerate}
\item[(i)] Does the middle-third Cantor measure admit a Fourier frame?

\item[(ii)] Do a triangle or a disk admit an exponential Riesz basis?
\end{enumerate}

\subsection{Exponential functions with non-linear phases}

Our main goal here is to consider a different question regarding frames
generated by measures, which we will always assumed to be bounded.

\medskip Let $\mu$ be a finite Borel measure with closed support $K_{\mu}$.
%we let $K_{\mu}$Then $X_{\mu}$ is called an \textit{essential support} of
%$\mu$ if it is a Borel set of $\mu$ such that $\mu(X_{\mu}) = \mu({\mathbb{R}%
%}^{d})$. and we will assume $X_{\mu}\subset K_{\mu}$.\todo{Comments by Vignon: This part may need to be ironed out a little bit}
Let $\varphi$ be a Borel measurable function defined on a set containing
$K_{\mu}$ and $\Lambda$ be a countable and discrete set. We define the
collection of generalized exponentials with phase functions $\varphi$ as
\[
E(\Lambda,\varphi) = \left\{  e^{2\pi i \lambda\cdot\varphi(x)}: \lambda
\in\Lambda\right\}  .
\]

We aim to address the following question:

\smallskip

\begin{ques}
\label{question} Let $\mu$ be a finite Borel measure on ${\mathbb{R}}^{d}$
with support $K_{\mu}$. Can we characterize the pairs $(\Lambda, \varphi)$ for
which the system $E(\Lambda,\varphi)$ is an orthonormal basis, or a frame for
$L^{2}\left(  \mu\right)  ?$
\end{ques}

\smallskip

Note that $E(\Lambda,\varphi)$ forms a frame if and only if (\ref{Eq_Frame})
holds when the exponential function $e^{2\pi i \lambda\cdot x}$ is replaced
with $e^{2\pi i \lambda\cdot\varphi(x)}.$ The motivation of the work is
manifold. On the one hand, we are are interested in the flexibility afforded
by $\varphi$ in forming exponential bases/frames for $L^{2}(\mu)$. On the
other hand, complex exponential systems with non-linear phases arise naturally
in the realization of unitary representations of non-commutative Lie groups.
Such systems have been investigated from a representation-theoretic viewpoint
in
\cite{grochenig2017orthonormal,oussa2018frames,oussa2017regular,oussa2018framesI}%
. The overarching theme in these projects is concerned with the question of
discretizing a unitary (irreducible) representation of a locally compact group
for the construction of frames and orthonormal bases. These questions are
deeply connected to wavelet theory, and time-frequency analysis
\cite{BeTa,Ole,Gr} and have been studied with varying degrees of generality
\cite{Gr,grochenig2017orthonormal,oussa2018frames,oussa2018framesI,MR809337}.

\subsection{Main Results and organization}

This paper aims to present a mathematical framework on the spectrality and
frame-spectrality of $E(\Lambda,\varphi)$ for a general finite Borel measure
$\mu$. Our main result gives a complete characterization of pairs
$(\Lambda,\varphi)$ for which the system of vectors $E(\Lambda,\varphi)$ forms
a frame or a basis for $L^{2}(\mu).$ The work is organized as follows.

\medskip

In the second section, we define the concept of essential injective Borel
measurable functions. This notion plays a central role in the proof of our
main result, which is stated in Theorem \ref{theorem_basis_frame}.
%in terms of the essential injectivity of $\varphi$ and the (frame-)spectrality of the push-forward map $\varphi_{\ast}\mu$.

\medskip

In Section 4, we will illustrate our main result with some specific examples.
In particular, we show that the middle-third Cantor measure admits an
orthonormal basis of the form $E(\Lambda,\varphi)$ for some Borel-measurable
function $\varphi$. Additionally, we also establish that the unit disc also
admits an orthonormal basis of the type $(\Lambda,\varphi)$ where $\varphi$ is
a piecewise continuously differentiable phase function. These examples clearly
illustrate how allowing the presence of $\varphi$ in the ``generalized''
exponentials  we consider greatly simplify the problem of constructing
orthonormal bases or frames in this setting.
%\todo{Comments by Vignon Would it be helpful to rephrase this sentence? Perhaps, we could say something along the following lines: T}

\medskip

In Section 5, we consider the problem of determining which continuous function
$\varphi$ have the property that the system $E({\mathbb{Z}}, \varphi)$ is an
orthonormal basis for $L^{2}[0,1]$. We show that if $\varphi$ preserves
measure zero sets, the only such $\varphi$ are exactly $\varphi(x) = \pm x$.

\medskip

In Section 6, we show that, in the multidimensional setting, we cannot expect
results analogous to those obtained in the one-dimensional case. We also
describe in Theorem \ref{unipotent} a fairly large class of functions
$\varphi$ such that $E({\mathbb{Z}}^{d},\varphi)$ is an orthonormal basis for
$L^{2}[0,1]^{d}$. When $d=2$, we present sufficient conditions under which we
cannot improve the construction given in Theorem \ref{unipotent}. \medskip

In Section 7, we explore applications to the discretization of unitary
representations of a class of Lie groups for the construction of orthonormal
bases
\cite{Gr,grochenig2017orthonormal,oussa2018frames,oussa2017regular,oussa2018framesI,MR809337}%
.

\medskip

Finally, in Section 8, we conclude our paper by stating several open
questions.

\medskip

\section{Essential injectivity}

In this section, we define and discuss the property of essential injectivity
of a Borel map. Let $\mu$ be a finite Borel measure on ${\mathbb{R}}^{d}$ with
support $K_{\mu}$.  Given a Borel measurable map $\varphi:K_{\mu
}\rightarrow\mathbb{R}^{d}$, the pushforward measure of $\mu$ under $\varphi$
is characterized by the property
\[
\varphi_{\ast}\mu(E)=\mu(\varphi^{-1}(E)),\quad E\subset{\mathbb{R}}^{d},
\quad E\ \mbox{Borel}.
\]
Moreover, integration with respect to the measure $\varphi_{\ast}\mu$
satisfies the following formula.
\[
\int f(y) d(\varphi_{\ast}\mu)(y) = \int f(\varphi(x))d\mu(x), \quad f\in
C_{0}( {\mathbb{R}}^{d}),
\]
where $C_{0}( {\mathbb{R}}^{d})$ denotes the space of continuous functions on
${\mathbb{R}}^{d}$ that vanish at infinity.
%Given a countable set $\Lambda$, we are interested in characterizing Borel
%measurable maps $p$ for which the system of exponentials $\{e^{2\pi
%i\lambda\cdot p(x)}:\lambda\in\Lambda\}$ is complete/ONB/frame for $L^{2}%
%(\mu)$.
The concept of essential injectivity as defined below plays a central role in
this work.

\begin{definition}
Given a finite Borel measure $\mu$ with support $K_{\mu},$ a Borel
measurable function $\varphi:K_{\mu}\rightarrow{\mathbb{R}}^{d}$ is said to be
\textit{$\mu$-essentially injective} if for any $f\in L^{2}(\mu)$, there exists
$h\in L^{2}(\varphi_{\ast}\mu)$ such that $f(x)=h(\varphi(x))$ $\mu$-a.e.
\end{definition}

We will record some conditions equivalent to essential injectivity in the
following lemma.

\begin{lem}
\label{comp} Let $\varphi:K_{\mu}\rightarrow\mathbb{R}^{d}$ be Borel
measurable and let $\nu=\varphi_{\ast}\mu$ be the corresponding pushforward
measure. Then, the following are equivalent.

\begin{enumerate}
\item The span of the collection $\chi_{\varphi^{-1}(F)}$, $F\subset
{\mathbb{R}}^{d}$, $F$ Borel, is dense in $L^{2}(\mu)$.

\item For any $f\in L^{2}(\mu)$, there exists $h\in L^{2}(\nu)$ such that
$f(x)=h(\varphi(x))$ $\mu$-a.e .

\item There exists a Borel set $\mathcal{N}\subset K_{\mu}$ with
$\mu({\mathcal{N}}) = 0$ such that $\varphi$ is injective on $K_{\mu}%
\setminus{\mathcal{N}}$.
\end{enumerate}
\end{lem}

\begin{proof}
$((2)\Longrightarrow(1)).$ Letting $A\subset K_{\mu}$ be a Borel subset of
${\mathbb{R}}^{d},$ then $\chi_{A}$, the indicator function of $A$, satisfies
the following condition: There exists $h\in L^{2}(\nu)$ such that $\chi
_{A}(x)=h(\varphi(x))$ $\mu$-a.e. Furthermore, we may assume that $h\geq0$
$\nu-$a.e by replacing $h$ by $|h|$ if necessary. Noting that
\[
\nu(\{h\neq0\,\,\text{or}\,\,1\})=\mu(\{x\in K_{\mu},\,\,h(\varphi
(x))\neq0\,\,\text{or}\,\,1\})=\mu(\{x\in K_{\mu},\,\,\chi_{A}(x)\neq
0\,\,\text{or}\,\,1\})=0,
\]
we obtain that $h=\chi_{E}$ a.e. ($d\nu$) where $E=\{y\in{\mathbb{R}}%
^{d},\,\,h(y)=1\}$. Thus
\[
\chi_{A}(x)=\chi_{E}(\varphi(x))=\chi_{\varphi^{-1}(E)}(x)
\]
for $\mu$-a.e.~$x$. In other words, $\chi_{A}=\chi_{\varphi^{-1}(E)}$ and (1)
follows since the linear span of the characteristic functions of the Borel
subsets (i.e. all simple functions) forms a dense subspace in $L^{2}(\mu)$.

$((1)\Longrightarrow(2)).$ Suppose that (1) holds. For a fixed vector $f\in
L^{2}(\mu),$ there exists a sequence $\{f_{n}\}_{n\in\mathbb{N}}$ of the form
\[
f_{n}(x)=\sum_{j=1}^{M(n)}\,c_{n,j}\,\chi_{\varphi^{-1}(E_{n,j})}%
(x)=\sum_{j=1}^{M(n)}\,c_{n,j}\,\chi_{E_{n,j}}(\varphi(x)),\quad n\geq1,
\]
where $c_{n,j}\in\mathbb{C}$ and $E_{n,j}$ are Borel subset of $\mathbb{R}%
^{d}$ for $n\geq1$ and $1\leq j\leq M(n)$, such that $f_{n}\rightarrow f$ in
$L^{2}(\mu)$. Letting
\[
h_{n}(y)=\sum_{j=1}^{M(n)}\,c_{n,j}\,\chi_{E_{n,j}}(y),\quad n\geq
1,\,\,\left(  y\in\mathbb{R}^{d}\right)  ,
\]
we obtain
\[
\int_{\mathbb{R}^{d}}\,|h_{n}(y)-h_{m}(y)|^{2}\,d\nu(y)=\int\,|h_{n}%
(\varphi(x))-h_{m}(\varphi(x))|^{2}\,d\mu(x)=\int\,|f_{n}(x)-f_{m}%
(x)|^{2}\,d\mu(x).
\]
Since $\{f_{n}\}_{n\in\mathbb{N}}$ is a Cauchy sequence in $L^{2}(\mu)$, it
follows that $\left\{  h_{n}\right\}  _{n\in\mathbb{N}}$ is a Cauchy sequence
in $L^{2}(\nu)$. Since $L^{2}(\nu)$ is complete, there exists a vector $h\in
L^{2}(\nu)$ such that $h_{n}\rightarrow h$ in $L^{2}(\nu)$. Consequently,
\begin{align*}
& \int\,|f(x)-h(\varphi(x))|^{2}\,d\mu(x)\\
&  \leq2\,\left\{  \int\,|f(x)-h_{n}(\varphi(x))|^{2}\,d\mu(x)+\int
_{[0,1)^{d}}\,|h_{n}(\varphi(x))-h(\varphi(x))|^{2}\,d\mu(x)\right\} \\
&  =2\,\left\{  \int\,|f(x)-f_{n}(x))|^{2}\,d\mu(x)+\int_{\mathbb{R}^{d}%
}\,|h_{n}(y))-h(y)|^{2}\,d\nu(y)\right\}
\end{align*}
and the last quantity converges to zero as $n \to \infty$. This shows that
$f(x)=h(\varphi(x))$, for $\mu$-a.e.

\medskip

$((2)\Longrightarrow(3).)$ Assuming that (2) holds, for fixed $k\in\left\{
1,\cdots,d\right\} ,$ there exist functions $h_{k}$ in $L^{2}(\nu)$ such that
$x_{k}=h_{k}(\varphi(x))$, for $\mu$-a.e.~$x\in\mathbb{R}^{d}$. Note that $\nu$ is a finite Borel measure and hence regular. For a regular Borel measure, we can find Borel measurable function that are equal to $h_k$ $\mu$-a.e.   We may therefore assume that $h_k$ is Borel measurable. Defining
$\mathbf{h}=(h_{1},...,h_{d})$, we obtain
\begin{equation}
\label{eqx}x=\mathbf{h}(\varphi(x))
\end{equation}
for a.e.~$x\in K_{\mu}$. Consequently, letting ${\mathcal{N}}$ be the subset
of $K_{\mu}$ for which (\ref{eqx}) does not hold, it follows that $\mu(K_{\mu
}\setminus{\mathcal{N}}) = 0$. If $x,y\in K_{\mu}\setminus{\mathcal{N}}$ and
$\varphi(x) = \varphi(y)$, applying (\ref{eqx}), we obtain that $x=y$. Hence,
$\varphi$ is injective on $K_{\mu}\setminus{\mathcal{N}}$. Finally, since $h_k$ and $\varphi$ are Borel measurable, so is $g = h\circ\varphi- I$, where $I$ is the identity map.  ${\mathcal N}$ is therefore  a Borel set since we can write
$$
K_{\mu}\setminus{\mathcal N} = \{g=0\} = \bigcap_{n=1}^{\infty} \left\{-\frac1{n}<g<\frac1n\right\}.
$$
. 

\medskip

$((3)\Longrightarrow(2).)$ As $\varphi$ is Borel measurable and $\varphi
|_{K_{\mu}\setminus{\mathcal{N}}}$ is injective, by \cite[Corollary
15.2]{Kechris}, $\varphi(B)$ is Borel measurable for all Borel sets $B\subset
K_{\mu}\setminus{\mathcal{N}}$. Hence, we can define $g = \varphi^{-1}:
\varphi[K_{\mu}\setminus{\mathcal{N}}]\to K_{\mu}\setminus{\mathcal{N}}$ such
that $x = g(\varphi(x))$ and $g$ is Borel measurable. Given any $f\in
L^{2}(\mu)$, since $\mu$ is a regular Borel measure, we can find a Borel
measurable function $\widetilde{f}$ such that $f(x) = \widetilde{f}(x)$ $\mu
$-a.e. We also have $\widetilde{f}(x) = \widetilde{f} (g(\varphi(x)))$ on
$K_{\mu}\setminus{\mathcal{N}}$ (thus $\mu$-a.e.). Define $h = \widetilde
{f}\circ g$. Then $f(x) = \widetilde{f}(x) = h(\varphi(x))$ $\mu$ a.e, $h$ is
Borel measurable and
\[
\int|h|^{2}d\nu= \int|\widetilde{f}(g(\varphi(x)))|^{2}d\mu(x) =
\int|\widetilde{f}(x)|^{2}d\mu(x) = \int|f|^{2}d\mu.
\]
Therefore, $h\in L^{2}(\nu)$ and the proof is complete.
\end{proof}

\medskip

\section{Main Characterization}

The aim of this section is to provide necessary and sufficient conditions for
the completeness, spectrality and frame-spectrality of $E(\Lambda,\varphi)$ in
terms of the essential injectivity of $\varphi$ and the corresponding property
of the system $E(\Lambda)$ in the $L^{2}$-space associated with the
push-forward measure.

\subsection{Completeness.}

We recall that a collection $\{f_{n}\}$ is complete in the Hilbert space $H$
if $\langle f, f_{n}\rangle= 0$ for all $n$ implies that $f = 0$ on $H$. This
condition is also equivalent to the fact that the closure of the linear span
of $f_{n}$ is dense in $H$. We notice that an orthogonal basis must be a
frame, and a frame must be complete in $H$.

\begin{theorem}
\label{theorem_complete} Let $\mu$ be a finite Borel measure on ${\mathbb{R}%
}^{d}$, let $\Lambda\subset\mathbb{R}^{d}$ be countable and let $\varphi
:K_{\mu}\rightarrow\mathbb{R}^{d}$ be Borel measurable. Let $\nu$ be the
pushforward Borel measure on $\mathbb{R}^{d}$ associated with $\varphi$. Then,
the following are equivalent.

\begin{enumerate}
\item The collection $E(\Lambda, \varphi)$ is complete in $L^{2}(\mu)$.

\item $\varphi$ is $\mu$-essentially injective, and $E(\Lambda)$ is complete in
$L^{2}(\nu)$.
\end{enumerate}
\end{theorem}

\begin{proof}
Let us assume that (2) holds and let $g\in L^{2}(\mu)$ such that
\[
\int\,g(x)\,e^{-2\pi i\lambda\cdot\varphi(x)}\,d\mu(x)=0,\quad\lambda
\in\Lambda.
\]
We aim to show that $g$ must be the trivial element in $L^{2}(\mu).$ Since
$\varphi$ is $\mu$-essentially injective, appealing to Lemma \ref{comp}, there
exists $h\in L^{2}(\nu)$ such that $g(x)=h(\varphi(x))$ $\mu$-a.e. As such,
\[
0=\int\,h(\varphi(x))\,e^{-2\pi i\lambda\cdot\varphi(x)}\,d\mu(x)=\int
_{\mathbb{R}^{d}}\,h(y)\,e^{-2\pi i\lambda\cdot y}\,d\nu(y),\quad\lambda
\in\Lambda.
\]
By assumption, $\{e^{2\pi i\lambda\cdot x}\}_{\lambda\in\Lambda}$ is complete
in $L^{2}(\nu),$ and it follows that $h=0$. Consequently,
\[
\Vert g\Vert_{2,\mu}^{2}=\int\left\vert h(\varphi(x))\right\vert ^{2}%
d\mu\left(  x\right)  =\int\left\vert h(y)\right\vert ^{2}d\nu\left(
y\right)  =\Vert h\Vert_{2,\nu}^{2}=0
\]
and we conclude that $g=0$ as desired.

\medskip

Conversely, let $\xi\in\mathbb{R}^{d}$ and fix $\epsilon>0$. Since simple
functions are dense in $L^{2}(\nu)$, there exist Borel measurable sets
$F_{k}\subset\mathbb{R}^{d}$ and $c_{k}\in\mathbb{C}$, $k=1,\dots,m$, such
that
\[
\Vert e^{2\pi i\xi\cdot x}-\sum_{k=1}^{m}\,c_{k}\,\chi_{F_{k}}(x)\big\|_{L^{2}%
(\nu)}<\epsilon.
\]
This implies that
\[
\int\,\big|e^{2\pi i\xi\cdot\varphi(x)}-\sum_{k=1}^{m}\,c_{k}\,\chi
_{\varphi^{-1}(F_{k})}(x)\big|^{2}d\mu(x)=\int\big|e^{2\pi i\xi\cdot
\varphi(x)}-\sum_{k=1}^{m}\,c_{k}\,\chi_{F_{k}}(\varphi(x))\big|^{2}%
d\mu(x)<\epsilon^{2}.
\]
Since we assume that the span of the collection $E(\Lambda, \varphi) =
\{e^{2\pi i\lambda\cdot\varphi(x)}\}_{\lambda\in\Lambda}$ is dense in
$L^{2}(\mu)$, the above shows that the span of the collection $\chi
_{\varphi^{-1}(F)}$, where $F\subset\mathbb{R}^{d}$ is Borel, is also dense
and therefore $\varphi$ is $\mu$-essentially injective by Lemma \ref{comp}. We
finally show that $E(\Lambda)$ is complete in $L^{2}(\nu)$. To see this, let
$f\in L^{2}(\mu)$ and suppose that $\int f(x) e^{-2\pi i \lambda\cdot x}%
d\nu(x) = 0$ for all $\lambda\in\Lambda$. Then
\[
\int f(\varphi(x)) e^{-2\pi i \lambda\cdot\varphi(x)}d\mu(x) = 0,
\ \forall\lambda\in\Lambda.
\]
By (1), $f(\varphi(x)) = 0$ $\mu$-a.e. As a set, we note that $\varphi
^{-1}\{y: f(y)\ne0\} = \{x: f(\varphi(x))\ne0\}$.
\[
\nu(\{f\ne0\}) = \mu(\varphi^{-1}\{f\ne0\}) = \mu(\{x: f(\varphi(x))\ne0\}) =
0.
\]
This shows $f=0$, $\nu$-a.e, which means that $E(\Lambda)$ is complete in
$L^{2}(\nu)$.
\end{proof}

\medskip

\subsection{Orthogonal basis and frames}

We shall now give a complete characterization of functions $\varphi: K_{\mu
}\to\mathbb{R}^{d}$ for which $E(\Lambda,\varphi)$ forms an orthogonal basis
or a frame for $L^{2}(\mu)$.

%\textcolor{red}{Do we want to give a formal definition of a frame?}

\begin{theorem}
\label{theorem_basis_frame} Let $\mu$ be a finite Borel measure,
$\Lambda\subset\mathbb{R}^{d}$ be countable and let $\varphi: K_{\mu}%
\to\mathbb{R}^{d} $ be Borel measurable. Let $\nu$ be the pushforward Borel
measure on $\mathbb{R}^{d}$ associated with $\varphi$ and $\mu$ Then

\begin{enumerate}
\item The collection $E(\Lambda, \varphi)$ is an orthogonal basis for
$L^{2}(\mu)$ if and only if $\varphi$ is $\mu$-essentially injective and $\nu$ is a
spectral measure with a spectrum $\Lambda$.

\item The collection $E(\Lambda,\varphi)$ forms a frame for $L^{2}(\mu)$ if
and only if $\varphi$ is $\mu$-essentially injective and $\nu$ is a frame-spectral
measure with a frame spectrum $\Lambda$.
%\begin{enumerate}
%\item $\nu$ is a spectral measure with $\Lambda$ as a spectrum.
%\item The span of the collection $\chi_{\varphi^{-1}(F)}$, $F\subset \mathbb{R}^d$, $F$ Borel,
%is dense in $L^2([0,1)^d)$.
%\end{enumerate}

\end{enumerate}
\end{theorem}

%\textcolor{red}{Should we define the concepts of spectral measure and frame-spectral measure?}

\begin{proof}
For the first part, we note that the pushforward Borel measure $\nu$ is
bounded and
\[
\int_{\mathbb{R}^{d}}\,f(y)\,d\nu(y)=\int_{}\,f(\varphi(x))\,d\mu(x),
\]
for any bounded, continuous function $f:\mathbb{R}^{d}\to\mathbb{C}$. The
orthogonality in $L^{2}(\nu)$ of the collection of exponentials with spectrum
in $\Lambda$ gives
\[
\int\,e^{2\pi i (\lambda-\lambda^{\prime}\cdot\varphi(x)) }\,d\mu(x)
=\int_{\mathbb{R}^{d}}\,e^{2\pi i (\lambda-\lambda^{\prime})\cdot y }%
\,d\nu(y)=\delta_{\lambda,\lambda^{\prime}}, \quad\lambda,\lambda^{\prime}%
\in\Lambda.
\]
This shows that the collection $\{e^{2\pi i \lambda\cdot\varphi(x)
}\}_{\lambda\in\Lambda}$ is mutually orthogonal for $L^{2}(\mu)$ if and only
if $\{e^{2\pi i \lambda\cdot x }\}_{\lambda\in\Lambda}$ is mutually orthogonal
for $L^{2}(\nu)$. Next, appealing to Theorem \ref{theorem_complete}, we obtain
that $E(\Lambda, \varphi)$ is complete if and only if $\varphi$ is $\mu$-essentially
injective and $E(\Lambda)$ is complete in $L^{2}(\nu).$ This takes care of the
first part of the result.

\medskip

\noindent For the second part, let us suppose that $E(\Lambda,\varphi)$ forms
a frame for $L^{2}(\mu)$. As such, the collection $E(\Lambda, \varphi)$ is
complete in $L^{2}(\mu)$. By Theorem \ref{theorem_complete}, ${\varphi}$ is
$\mu$-essentially injective. Next, given $h\in L^{2}(\nu)$,
\[
\sum_{\lambda\in\Lambda}\left\vert \int h(y)e^{-2\pi i\lambda\cdot y}%
d\nu(y)\right\vert ^{2}=\sum_{\lambda\in\Lambda}\left\vert \int h(\varphi
(x))e^{-2\pi i\lambda\cdot\varphi(x)}d\mu(y)\right\vert ^{2}.
\]
However,
\[
\int|h(y)|^{2}d\nu(y)=\int|h(\varphi(x))|^{2}d\mu(x),
\]
and the frame inequalities for $E(\Lambda)$ in $L^{2}(\nu)$ follow from those
of $E(\Lambda, \varphi)$ in $L^{2}(\mu)$. Conversely, since $\varphi$ is
$\mu$-essentially injective, any function $f\in L^{2}(\mu)$ can be written as
$f(x)=h(\varphi(x))$ $\mu$-a.e. with $h\in L^{2}(\nu)$. Therefore, the frame
inequalities for $E(\Lambda, \varphi)$ for $L^{2}(\mu)$ also follow from those
of $E(\Lambda)$ in $L^{2}(\nu)$.
\end{proof}

\medskip

\subsection{Lebesgue measures on general sets.}

We now turn to study the special case where $\mu$ is the Lebesgue measure
restricted on a set. Given a Borel set $\Omega\subseteq{\mathbb{R}}^{d}$ of
positive finite Lebesgue measure, we define the measure $m_{\Omega}$ as
follows:%
\[
m_{\Omega}=\chi_{\Omega}\,dx
\]

\medskip

Letting $\Lambda$ be a countable discrete set, we define the \textit{upper and
lower Beurling density} of $\Lambda$, denoted by $D^{+}(\Lambda)$ and
$D^{-}(\Lambda)$ respectively, as follows:
\[
D^{+}(\Lambda)=\limsup_{R\rightarrow\infty}\sup_{x\in{\mathbb{R}}^{d}}%
\frac{\#(\Lambda\cap Q_{R}(x))}{R^{d}},\ \mbox{and}\ D^{-}(\Lambda
)=\liminf_{R\rightarrow\infty}\inf_{x\in{\mathbb{R}}^{d}}\frac{\#(\Lambda\cap
Q_{R}(x))}{R^{d}},
\]
where $Q_{R}(x)=x+[-\frac{R}{2},\frac{R}{2}]^{d}$ denotes the hypercube of
side length $R$ centered at $x$. In \cite[Proposition 2.1]{HLL}, He, Lai and
Lau proved

\begin{proposition}
If $D^{-}(\Lambda)>0$ and $\Lambda$ is a frame spectrum for $L^{2}(\mu)$, then
$\mu$ must be absolutely continuous with respect to the Lebesgue measure.
\end{proposition}

Moreover, in
  \cite[Theorem 1.3 and Corollary 1.4]{DL14}, the following was
proved.
%\textcolor {red}{Note that the second statement below does not make sense. What is $\mu$ in relation to $\Omega$
%and what is $X_\mu$?}

\begin{theorem}
\label{theorem_DL} Let $\Lambda$ be a countable discrete set with
$D^{-}(\Lambda)>0$ and let $\mu$ be a finite Borel measure. The following
holds true

\begin{enumerate}
\item If $\{e^{2\pi i\lambda\cdot x}:\lambda\in\Lambda\}$ forms a frame for
$L^{2}(\mu)$, then $\mu$ is absolutely continuous with respect to the Lebesgue
measure with density $g$ satisfying $0<m\leq g\leq M<\infty$ almost everywhere
on the set $\{g\ne 0\}$. (for some positive real numbers $m$ and $M.$)

\item If $\{e^{2\pi i\lambda\cdot x}:\lambda\in\Lambda\}$ forms an orthogonal
basis for $L^{2}(\mu)$, then $\mu$ is absolutely continuous with respect to
the Lebesgue measure with a constant density. \end{enumerate}
\end{theorem}

\medskip

We now have the following characterization.
%\begin{lemma}
%Let $\varphi:\Omega\to \R$ be $\mu$-essentially injective and let $N$ be the measure zero set such that $\varphi$ is injective on $\Omega\setminus N$.  Suppose that $\nu = \varphi_{\ast}m_{\Omega}$. Then we can take $X_{\nu} = p(\Omega\setminus N)$ and $X_{\nu} = \varphi(\Omega)$  if $\varphi$ maps measure zero sets to a measures zero set.
%\end{lemma}
%\begin{proof}
%As $\varphi$ is injective on $\Omega\setminus N$, we have that $\varphi^{-1}(\varphi(\Omega\setminus N)) = \Omega\setminus N$. Hence, for any Borel set $E$,
%$$
%\nu((\varphi(\Omega\setminus N))\cap E) =m_{\Omega} ((\Omega\setminus N)\cap \varphi^{-1}(E)) = m_{\Omega} (\varphi^{-1}(E)) = \nu(E).
%$$
%Therefore, we can take $X_{\nu} = \varphi(\Omega\setminus N$. If $\varphi$ preserves measure zero set, then  $m(\varphi(\Omega\setminus N)) = m(\varphi(\Omega)\setminus \varphi(N)) = m(\varphi(\Omega))$.
%\end{proof}

%
%\textcolor{red}{
%I think that there is a bit of a problem
%with the statement of the theorem below.
%Note that even if $\varphi$ is continuous
%and $\Omega$ is Borel, $\varphi(\Omega)$ might not be
%Borel. However, $\varphi(\Omega)$ is Borel if
%$\varphi$ is injective which is almost what we have here.
%So I think that, in the statements below,
%we should replace  $\varphi(\Omega)$
%by  $\varphi(\Omega_0)$ where $\Omega_0$
%is a Borel subset of  $\Omega$ such that
%$\mu(\Omega\setminus\Omega_0 )=0$
%and $\varphi$ is injective on $\Omega_0$
%(and we know that such set $\Omega_0$ exists).
%}

\begin{theorem}
\label{Fourier_Frame} Let $\mu=m_{\Omega}$ where $\Omega$ is a measurable set
of finite positive Lebesgue measure and let $\Lambda\subset\mathbb{R}^{d}$ be
countable with $D^{-}(\Lambda)>0$. Given a Borel measurable function $\varphi:
K_{\mu}\rightarrow\mathbb{R}^{d} $, define $\nu=\varphi_{\ast}\mu$ to be the
pushforward of the measure $\mu$ via $\varphi.$ Then, the following holds.

\medskip

\begin{enumerate}
\item The collection $\{e^{2\pi i \lambda\cdot\varphi(x)}\}_{\lambda\in
\Lambda}$ is an orthogonal basis for $L^{2}(\mu)$ if and only if there exists a set $\Omega_0$ such that $\varphi$ is injective on $\Omega_0$ and $m(\Omega\setminus\Omega_0) = 0$ with

\begin{enumerate}
\item $\varphi$ is  injective on $\Omega_0$,

\item $\varphi_{\ast}m_{\Omega_0} = c\cdot m_{\varphi(\Omega_0)}$ for some $c>0$.

\item $\varphi(\Omega_0)$ is a spectral set with a spectrum $\Lambda$ .
\end{enumerate}

\item The collection $\{e^{2\pi i \lambda\cdot\varphi(x) }\}_{\lambda
\in\Lambda}$ forms a frame for $L^{2}(\mu)$ if and only if there exists a set $\Omega_0$ such that $\varphi$ is injective on $\Omega_0$ and $m(\Omega\setminus\Omega_0) = 0$ with

\begin{enumerate}
\item $\varphi$ is  injective on $\Omega_0$,

\item There exists $0<m\leq M<\infty$ such that $\varphi_{\ast}\mu=g(x)dx$
with $m\leq g\leq M$ a.e. on $\varphi(\Omega_0)$ and

\item $\{e^{2\pi i \lambda\cdot x}\}_{\lambda\in\Lambda}$ forms a Fourier
frame for $L^{2}(\varphi(\Omega_0))$.
\end{enumerate}
\end{enumerate}

%Moreover, if $\varphi\in C^{1}(\Omega)$ and is injective, then
%\begin{enumerate}
%\item[(i)]  Condition 1(b) holds if and only if the Jacobian of $\varphi$, $J(\varphi)$, has a constant
%determinant a.e.
%\item[(ii)]  Condition 2(b) holds if and only if $|\det J(\varphi)|$ is bounded above and bounded away from $0$
%a.e.
%\end{enumerate}

\end{theorem}

\begin{proof}
In order to prove the first statement, note that, by Theorem
\ref{theorem_basis_frame}, the fact that the collection $\{e^{2\pi i
\lambda\cdot\varphi(x)}\}_{\lambda\in\Lambda}$ forms an orthogonal basis for
$L^{2}(\mu)$ is equivalent to the essential injectivity of $\varphi$ and the
fact that $\varphi_{\ast}m_{\Omega} $ is a spectral measure with spectrum
$\Lambda$. Thus the conditions in 1(a), 1(b) and 1(c) clearly imply that
$\{e^{2\pi i \lambda\cdot\varphi(x)}\}_{\lambda\in\Lambda}$ is an orthogonal
basis for $L^{2}(\mu)$.

Conversely, according to Theorem \ref{theorem_basis_frame} again, if that same
collection is an orthogonal basis for $L^{2}(\mu)$, $\varphi$ must be
$\mu$-essentially injective, i.e.~1(a) holds and denote by $\Omega_0$ to be the set of full measure so that $\varphi$ is injective on. Note that the second condition in Theorem \ref{theorem_basis_frame} implies that  $\varphi_{\ast}m_{\Omega_0}$ must be
a spectral measure with spectrum $\Lambda$. Since we are also assuming that
$D^{-}(\Lambda)>0$, part (2) of Theorem \ref{theorem_DL} then shows that, for
some $c>0$, $\varphi_{\ast}m_{\Omega_0} = c\cdot m_{D}$, where $D$ is an
essential support of $\varphi_{\ast}m_{\Omega_0}$. But it is direct to see that $\varphi(\Omega_0)$ is a support of $\varphi_{\ast}m_{\Omega_0}$ since 
\[
\varphi_{\ast}m_{\Omega_0} (\varphi(\Omega_0))= m_{\Omega_0} (\varphi^{-1}%
\varphi(\Omega_0)) = m(\Omega_0) =
\varphi_{\ast}m_{\Omega_0}({\mathbb{R}}^{d}).
\]
% Hence 1(b) and 1(c) hold. if
%we can show that $\varphi(\Omega_0)$ is an essential support of $\varphi_{\ast
%}m_{\Omega_0}$. Indeed, since $\varphi^{-1}(\varphi(\Omega))\supset\Omega$ and
%$\varphi^{-1}({\mathbb{R}}^{d}) = K_{\mu}\supset\Omega$, the essential support
%property follows from

\medskip

The proof of (2) is similar to (1) by invoking Theorem \ref{theorem_DL} (2).
We will omit the details.
%\medskip
%Regarding the second part of (2), let us assume that $\varphi\in
%C^{1}(\Omega)$ and is injective. Let $f$ be a positive function in
%$L^{1}\left(  d\varphi_{\ast}m_{\Omega}\right)  .$ Well known change of variables
%procedures give the following
%\[
%\int f(x)d\varphi_{\ast}m_{\Omega}(x)=\int_{\Omega}f(\varphi(x))dx=\int_{\varphi(\Omega
%)}f(x)|\det J(\varphi^{-1})(x)|dx.
%\]
%Therefore, $\varphi_{\ast}m_{\Omega}=|\det J(\varphi^{-1})(x)|dx$ and our conclusion
%follows from Theorem \ref{theorem_DL}.

\end{proof}

It turns out that the condition on the Beurling density cannot be removed. If
we do not restrict the density of $\Lambda$, the measure $\varphi_{\ast
}m_{\Omega}$ may be singular with respect to the Lebesgue measure and the
space $L^{2}(\Omega)$  could actually admit an orthonormal basis of
exponentials with non-linear phases associated with a spectrum $\Lambda$
possessing a zero Beurling density (See Example \ref{example1}).

\section{Examples and Illustration}

\label{Fractal} In this section, we present some examples of exponentials with
non-linear phases that can form either a basis or a a frame for the Hilbert
space $L^{2}(\mu)$.
%Indeed, we use a Borel map $\varphi$, we can actually construct many examples.

\subsection{Fractal examples}

Recall that the standard middle-fourth Cantor measure $\nu_{4}$ is the unique
measure supported on the middle-fourth Cantor set $K_{4}$ satisfying the
self-similar identity.
\[
\nu_{4}(E) = \frac12\nu_{4}(4E)+\frac12\nu_{4}(4E-2), \quad E\subset
\mathbb{R},\,\, \ E\,\,\mbox{Borel}.
\]
For this self-similar measure, Jorgensen and Pedersen \cite{JP1998} found that
$L^{2}(\nu_{4})$ admits an orthonormal basis basis of exponentials $\{e^{2\pi
i \lambda x}:\lambda\in\Lambda\}$, where
\[
\Lambda_{4}= \left\{  \sum_{i=0}^{n-1} 4^{i} a_{i}: a_{i} \in\{0,1\}, n =
1,2,3...\right\} ,
\]
showing thus that  $\nu_{4}$ is a spectral measure.  The support of $\nu_{4}$,
$K_{4}$, can be expressed  as the set
\[
K_{4} = \bigcap_{n=1}^{\infty} \bigcup_{j=1}^{2^{n}} I_{j,n}
\]
where $I_{j,n}$ are the basic intervals of the Cantor set.

\begin{example}
\label{example1} Let $\mu$ be the Lebesgue measure on $[0,1]$ and let
$\varphi:[0,1]\to K_{4}$ be the map
\[
\varphi\left(  \sum_{i=1}^{\infty} \frac{\epsilon_{i}}{2^{i}}\right)  =
\sum_{i=1}^{\infty} \frac{2\epsilon_{i}}{4^{i}},
\]
where $\epsilon_{i}\in\{0,1\}$. This map is well-defined except on the set of
dyadic rational numbers, which has measure zero. We note that for each
$n^{\mathrm{th}}$ basic interval $I_{j,n}$, $\varphi^{-1}(I_{j,n})$ is a
dyadic interval. The collection of all the preimages of the $n^{\mathrm{th}}$
basic intervals, $\varphi^{-1}(I_{j,n})$, $j=1,...,2^{n}$ are exactly all the
dyadic intervals at the $n^{\mathrm{th}}$ stage. Therefore, Lemma
\ref{comp}(1) is satisfied, and thus, $\varphi$ is $\mu$-essentially injective.
Moreover, as the measure $\nu_{4}$ and $\mu$ are completely determined by
 their values  on the basic intervals, we must have
\[
\varphi_{\ast}\mu= \nu_{4}.
\]
Hence, by Theorem \ref{theorem_basis_frame}, we have that
\[
\{e^{2\pi i \lambda\varphi(x)}: \lambda\in\Lambda_{4}\}
\]
is an orthonormal basis for $L^{2}[0,1]$.
\end{example}

\medskip

\begin{example}
(\textbf{Non-linear phased exponential bases for middle-third Cantor
measures}) Let $K_{3}$ be the middle-third Cantor set and $\nu_{3}$ be the
middle-third Cantor measure, which can be defined analogously by replacing $4$
with $3$ in the middle-fourth Cantor set definition. Jorgenesen and Pedersen
\cite{JP1998} proved that there is no exponential orthonormal basis for
$L^{2}(\nu_{3}). $ Nonetheless, we can define
\[
\varphi\left(  \sum_{i=1}^{\infty} \frac{\epsilon_{i}}{3^{i}}\right)  =
\sum_{i=1}^{\infty} \frac{\epsilon_{i}}{4^{i}},
\]
Using a similar proof in Example \ref{example1}, we obtain that we have that
\[
\{e^{2\pi i \lambda\varphi(x)}: \lambda\in\Lambda_{4}\}
\]
is an orthonormal basis for $L^{2}(\nu_{3})$. This example was also observed earlier in \cite{BK2010}.
\end{example}

\medskip

\subsection{Unit balls.}

Constructing an exponential Riesz basis for the open unit disc ${\mathbb{D}}$
has been a challenging problem in basis and frame theory. We are however, able
to offer an explicit non-linear phased exponential family forming an
orthonormal basis for $L^{2}({\mathbb{D}})$. Let $A,B $ be open sets on
${\mathbb{R}}^{d}.$ We say that $\varphi: A\to B$ is a \textit{$C^{1}%
$-diffeomorphism} if $\varphi$ is a $C^{1}$ bijective map from $A$ to $B$.
$\varphi$ is called \textit{measure-preserving} if $\varphi_{\ast}m_{A} =
m_{B}$. We need the following simple proposition.

\begin{proposition}
\label{change-of-variable} Let $A,B$ be open sets on ${\mathbb{R}}^{d}$ and
let $\varphi:A\to B$ be a $C^{1}$diffeomorphism. Let $J(\varphi)$ be the
Jacobian matrix of $\varphi$. Then $\varphi_{\ast}m_{A}= m_{B}$ if and only if
$|\det J(\varphi)| = 1$.
\end{proposition}

\begin{proof}
We note that the change of variable formula is now valid. Therefore, we have
\[
\varphi_{\ast}m_{A} (E) = \int_{A} \chi_{E}(\varphi(x)) dx =\int\chi
_{E}(\varphi(x)) \chi_{B}(\varphi(x)) dx = \int\chi_{E\cap B}(x) |\det
J(\varphi(x))|^{-1}dx.
\]
Hence, if $|\det J(\varphi)| = 1$, then $\varphi_{\ast}m_{A}= m_{B}$ holds.
Conversely, if $\varphi_{\ast}m_{A}= m_{B}$ holds, then
\[
\int\chi_{E\cap B}(x) \left(  |\det J(\varphi(x))|^{-1} -1\right)  = 0
\]
for all Borel sets $E$. Hence, $|\det J(\varphi(x))|^{-1} =1$ holds a.e. on $B$.
\end{proof}

\medskip

It is possible to have a measure preserving $C^{1}$ diffeomorphism between
${\mathbb{D}}$ and the $\ell^{1}$ ball (which is a square) with the same area
as ${\mathbb{D}}$. The map was given by \cite[p.160]{Holhos} in which
$\varphi(x,y) = (X,Y)$ where
\[
X = \mbox{sgn}(x) \frac{\sqrt{x^{2}+y^{2}}}{\sqrt{2\pi}} \left(  \frac{\pi}%
{2}+ \sin^{-1}\left(  \frac{x^{2}-y^{2}}{x^{2}+y^{2}}\right)  \right)  ,
\]
\[
Y = \mbox{sgn}(y) \frac{\sqrt{x^{2}+y^{2}}}{\sqrt{2\pi}} \left(  \frac{\pi}%
{2}- \sin^{-1}\left(  \frac{x^{2}-y^{2}}{x^{2}+y^{2}}\right)  \right)  .
\]
Note that if $x^{2}+y^{2}=1$, then $|X|+|Y| = \sqrt{\pi/2}$. Hence, the
$\ell^{1}$-ball formed has a measure of $\pi$. Using this map, Theorem
\ref{theorem_basis_frame} and the fact that square admits an exponential
orthonormal basis $E(\Lambda)$. We have thus proved that

\begin{proposition}
There exists a map $\varphi$ and a set $\Lambda$ such that $E(\Lambda
,\varphi)$ form an orthonormal basis for $L^{2}({\mathbb{D}})$.
\end{proposition}

\medskip

\section{Classification of continuous phase functions where $d=1$}

In previous sections, we characterized Borel measurable function $\varphi$ so
that $E(\Lambda,\varphi)$ forms a basis or frame. From this section on, we
will focus our attention on $\varphi$ being continuous functions or
differentiable functions. In particular, we are interested in the case where
$\Lambda= {\mathbb{Z}}^{d}$ and $E(\Lambda,\varphi)$ forms an orthogonal basis
for $[0,1]^{d}$.

\medskip

We will first study $d=1$ in this section. Here, we can work on $\varphi$
being continuous with a mild assumption that $\varphi$ preserves measure zero
sets. We may also assume that $\varphi(0) = 0$ since otherwise, the
exponentials will just be differing by a phase factor $e^{2\pi i \lambda
\cdot\varphi(0)}$, which will not affect the basis property. Our main result
is as follows:

\begin{theorem}
\label{p1} Let $\varphi:[0,1]\rightarrow{\mathbb{R}}$ be a continuous function
that maps Lebesgue measure zero sets to measure zero sets. Suppose also that
$\varphi(0)=0$. Then $E(\mathbb{Z},\varphi)$ is an orthonormal basis for
$L^{2}[0,1]$ if and only if $\varphi(x)=\pm x$.
\end{theorem}

We begin our proof with a general lemma.

\begin{lem}
\label{lemma_essential} Let $K$ be a bounded set of positive Lebesgue measure
and let $\varphi:K\rightarrow{\mathbb{R}}^{d}$ be a continuous function taking
measure zero sets to measure zero sets. Suppose there exists a set $E\subset
{\mathbb R}^d$, of positive measure satisfying $E=\varphi(U_{1})=\varphi(U_{2})$ with
$U_{1}\cap U_{2}=\emptyset$ and $U_{1},U_{2}$ are of positive Lebesgue measure
in $K$. Then $\varphi$ cannot be $\mu$-essentially injective.
\end{lem}

\begin{proof}
We argue by contradiction. Suppose that $\varphi$ is $\mu$-essentially injective.
Letting $f(x)=x$, we have then $f\in L^{2}(K)$. Since $\varphi$ is $\mu$-essentially
injective, one can find
 $h\in L^{2}(\nu)$,
where $\nu=\varphi_*\mu$,
  such that $x=h(\varphi(x))$
for a.e. $x\in K$. Let $E,U_{1},U_{2}$ be sets satisfying the conditions in
the statement of the lemma. Then
\[
{\mathcal{F}}_{i} = \{x\in U_{i}:h(\varphi(x))=x\}
\]
has full Lebesgue measure in $U_{i}$ for $i=1,2.$ Consider the following
sets:
\[
{\mathcal{K}}_{1}=\varphi({\mathcal{F}}_{1}),\ {\mathcal{K}}_{2}%
=\varphi({\mathcal{F}}_{2}).
\]
Since $\varphi$ maps measure zero sets to measure zero sets, ${\mathcal{K}%
}_{1}$ and ${\mathcal{K}}_{2}$ have full Lebesgue measure in $E$. Moreover,
$\varphi$ is also continuous, so ${\mathcal{K}}_{i}$ are measurable ( since we
can decompose ${\mathcal{F}}_{i}$ into $F_{\sigma}$-sets and measure zero
sets). Thus, $m({\mathcal{K}}_{1}\cap{\mathcal{K}}_{2}) =m( E) $. Next, given
$a\in{\mathcal{K}}_{1}\cap{\mathcal{K}}_{2}$, it is clear that $a=\varphi
(x)=\varphi(y)$ for some $x\in {\mathcal F}_{1}\subset U_1$ and $y\in {\mathcal F}_{2}\subset U_2$. Hence,
\[
x=h(\varphi(x))=h(a)=h(\varphi(y))=y,
\]
and this contradicts the fact that $U_{1}$ and $U_{2}$ are disjoint sets. This
completes the proof of the lemma.
\end{proof}

\begin{proposition}
\label{injective} Let $\varphi:[0,1]\rightarrow{\mathbb{R}}$ be a continuous
function mapping measure zero sets to measure zero sets. If $\varphi$ is
$\mu$-essentially injective, then $\varphi$ is injective.
\end{proposition}

\begin{proof}
We first observe that $\varphi$ cannot be a constant function on any
non-degenerate subinterval $I\subset\lbrack0,1]$. Otherwise, let $h\in L^{2}(\nu)$,
where $\nu=\varphi_*\mu$,
 such that $x=h(\varphi(x))$ a.e. on $[0,1]$. If
$\varphi(x)=c$ for $x\in I$, we would obtain that $h(c)=x$ for almost every
$x\in I\subset[0,1]$, clearly a contradiction.

\medskip

Suppose that $\varphi$ is not injective. There exist $x_{0},y_{0}$ such that
$x_{0}\neq y_{0}$ and $\varphi(x_{0})=\varphi(y_{0})=c$. As $\varphi$ is
continuous, given any $\epsilon>0$, we can find $\delta>0$ such that
$\varphi(x)\in(c-\epsilon,c+\epsilon)$ and $\varphi(y)\in(c-\epsilon
,c+\epsilon)$ whenever $|x-x_{0}|<\delta$ and $|y-y_{0}|<\delta$. Let
$x_{1}\in(x_{0}-\delta,x_{0}+\delta)$ and $y_{1}\in(y_{0}-\delta,y_{0}%
+\delta)$. For sufficiently small $\delta$, we may assume that $(x_{0}%
-\delta,x_{0}+\delta)$ and $(y_{0}-\delta,y_{0}+\delta)$ are disjoint. Since $\varphi$ cannot be constant on any intervals, we may
assume that $\varphi(x_1),\varphi(y_1)$ are not equal to $c$ and  $\varphi(x_{1})<\varphi(y_{1})$. For
the subsequent analysis, we define $I_{x,y}$ and $I_{x,y}^{\circ}$ respectively as the closed and open interval with
endpoints $x,y$.

\medskip

\noindent\textbf{Case (1): $\varphi(x_{1})<\varphi(y_{1})<c$ or $c<\varphi
(x_{1})<\varphi(y_{1})$.} Since these two cases are symmetric, there is no
loss of generality in only addressing one of them. By the intermediate value
theorem applied to $\varphi$ defined on $I_{x_{1},x_{0}}$, $\varphi$ assumes
all values in the interval $[\varphi(x_{1}),c]$ from ${I_{x_{1}%
,x_{0}}}$. Hence,  $\varphi({I_{x_{1},x_{0}}}) \supset[\varphi(x_{1}%
),c]\supset [\varphi(y_1),c]$. Similarly, we also have $\varphi(I_{y_{1},y_{0}})
\supset[\varphi(y_{1}),c]$. Note that $[\varphi(y_{1}),c]$ is now a
non-degenerate subinterval of $[\varphi(x_{1}),c]$. Let $E=(\varphi(y_{1}),c)$
and let $U_{1}=I_{x_{1},x_{0}}^{\circ}\cap\varphi^{-1}(E)$, $U_{2}=I_{y_{1},y_{0}}^{\circ}%
\cap\varphi^{-1}(E)$. Then $U_{1},U_{2}$ has positive Lebesgue measure (since
they are open) and disjoint. Moreover, $\varphi(U_{1}) = \varphi(U_{2}) = E$.
All assumptions in Lemma \ref{lemma_essential} are satisfied. Hence, $\varphi$
cannot be $\mu$-essentially injective. This completes the proof of Case (1).

\medskip

\noindent\textbf{Case (2): $\varphi(x_{1})<c<\varphi(y_{1})$.} We may assume
that $\varphi[x_{0}-\delta,x_{0}+\delta]\subset\lbrack c-\epsilon,c]$ and
$\varphi[y_{0}-\delta,y_{0}+\delta]\subset\lbrack c,c+\epsilon]$. Otherwise,
we can select $x_{1},y_{1}$ to satisfy the assumptions in Case (1). We can
also assume that the endpoints do not take the value $c$, otherwise, we choose
a smaller $\delta$. By the intermediate value theorem, $\varphi(
{I_{x_{1},x_{0}}})\supset[\varphi(x_{1}),c]$. On the other hand, consider the
interval $(x_{0} +\delta,y_{0}-\delta)$ if $x_0<y_0$ and  $(y_{0} +\delta,x_{0}-\delta)$ if $y_0<x_0$. We only consider the first case since the case $y_0<x_0$ is similar.  Intermediate Value Theorem tells us
that all values in the interval $\varphi[x_{0}+\delta,y_{0}-\delta
]\supset[\varphi(x_{0}+\delta),\varphi(y_{0}-\delta)]$. Note that the interval $E=(\max\{\varphi(x_{1}),\varphi(x_{0}%
+\delta)\},c)$ is non-degenerate (since endpoints do not take the value $c$). And the set $U_{1}=I_{x_{1},x_{0}}^{\circ}\cap\varphi^{-1}(E)$ and $U_{2}%
=(x_{0}+\delta,y_{0}-\delta)\cap\varphi^{-1}(E)$ satisfies the assumption that $\varphi(U_1) = \varphi(U_2) = E$ with positive Lebesgue measure. The assumption in  
Lemma \ref{lemma_essential} are all satisfied, so $\varphi$ cannot be $\mu$-essentially injective.
This completes the proof.
\end{proof}

\noindent\textit{Proof of Theorem \ref{p1}.} The statement that $\varphi
(x)=\pm x$ implies that $\left\{  e^{2\pi ik\cdot\varphi(x)}:k\in{\mathbb{Z}%
}\right\}  $ is an orthonormal basis for $L^{2}[0,1]$ is evident and we shall
focus on its converse. So, let us assume that $E({\mathbb{Z}}^{d},\varphi)$ is
an orthonormal basis for $L^{2}[0,1].$ Then $\varphi$ must be $\mu$-essentially
injective by Theorem \ref{theorem_basis_frame}. Proposition \ref{injective}
implies that $\varphi$ must be injective. As $\varphi$ is continuous,
$\varphi$ is monotone.

Without loss of generality, we can assume $\varphi$ is increasing. Then
$\varphi[0,1] = [0,\varphi(1)]$. Hence, $\varphi_{\ast}m_{[0,1]} = c\cdot
m_{[0,\varphi(1)]}$ by Theorem \ref{theorem_basis_frame} 1(b). As
$E({\mathbb{Z}})$ is an exponential orthonormal basis for $m_{[0,\varphi(1)]}%
$, we must have the $\varphi(1) = m([0,\varphi(1)]) = 1$ and $c=1$. Thus, $t =
m_{[0,1]}(0,t) = \varphi_{\ast}m_{[0,1]}((0,t))=\varphi^{-1}(t)$ and it
follows that $\varphi(x)=x$. The case that $\varphi$ is decreasing is similar
and $\varphi(x) = -x$. \qquad$\Box$

\medskip

Theorem \ref{p1} guarantees that under a weak condition imposed on $\varphi$, linear functions are the only maps $\varphi$ for which the
system $E(\Lambda,\varphi)$ is an orthonormal basis for $L^{2}\left[
0,1\right]  $. As it is well-known that all Lipschitz functions map measure
zero sets to measure zero sets, so Theorem \ref{p1} applies.

\medskip

We finally remark that $\mu$-essentially injectivity and injectivity are not the
same concept even if $\varphi$ is a continuous function. In search of the
literature, we found that Foschini \cite{F1970} constructed a continuous
$\mu$-essentially injective function that is not monotone in any intervals via
Wiener process. We do not know if this function can produce an exponential
orthonormal basis with a non-linear phase. However, it shows that the
assumption about preservation of measure zero sets in Proposition \ref{injective} cannot be removed.

%
%\begin{remark}(The case where $\varphi$ is real-analytic)
%\label{p2} $\varphi:\mathbb{R}\rightarrow\mathbb{R}$ is a real analytic and
%Lebesgue measure-preserving function if and only if there exists a constant $C$ such
%that $\varphi(x)=\pm x+C.$ Moreover, assuming that $\varphi$ is real analytic, the system
%$\left\{  e^{2\pi ik\cdot \varphi(x)}:k\in{\mathbb{Z}}\right\}  $ is an orthonormal
%basis for $L^{2}\left[  0,1\right)  $ if and only if $\varphi(x)=\pm x+C$ for some
%real constant $C.$ Indeed, let $\varphi:\mathbb{R}\rightarrow\mathbb{R}$ be a real analytic function which is measure preserving. Under the assumption that $\left\{  e^{2\pi ik\varphi(x)}%
%:k\in{\mathbb{Z}}\right\}  $ is an orthonormal basis for $L^{2}\left[
%0,1\right)  ,$ it must the case that $\varphi$ is a Lebesgue measure preserving map.
%Consequently, the derivative of $\varphi$ satisfies the differential equation
%$dy=\pm dx$ and this is the case if and only if $\varphi\left(  x\right)  =\pm x+C$
%for some constant $C.$ For the converse, assuming that $\varphi\left(  x\right)
%=\pm x+C$ for some constant $C,$ $e^{2\pi ik\varphi(x)}=e^{2\pi iC}e^{\pm2\pi ikx}$
%and it is clear that the system $\left\{  e^{2\pi ik\varphi(x)}:k\in{\mathbb{Z}%
%}\right\}  $ is an orthonormal basis for $L^{2}\left[  0,1\right)  .$
%\end{remark}

\medskip

\section{Non-linear phase in higher dimensions}

\subsection{A sufficient condition.}

In this section, we will investigate the construction of exponential
orthogonal bases with non-linear phases in higher dimensions. Our
investigation reveals that in the multidimensional settings, we cannot expect
a result that is analogous to Theorem \ref{p1} (see Theorem \ref{unipotent}).

Our first result in this section provides a large class of functions $\varphi$
such that $E({\mathbb{Z}}^{d},\varphi)$ is an orthonormal basis for
$L^{2}\left(  [0,1]^{d}\right)  $for arbitrary $d.$

%we will investigate the extent to which we are able to extend
%the results described in Remark \ref{p2} to higher dimensions. Furthermore,
%we will pay a particular attention to the case where $d=2.$
%Recall that a function $f$ with domain an open set $U\subset\mathbb{R}^{d}$
%with range $\mathbb{R}$ is called real analytic if for each $u\in U,$ the
%function $f$ can be represented by some power series in some neighborhood of
%$u.$ Moreover, such a function is continuous and has continuous analytic
%partial derivatives of all orders.

\begin{theorem}
\label{unipotent} Let $\varphi:[0,1]^{d}\rightarrow\mathbb{R}^{d}$ and let $x
= (x_{1},...,x_{d})$ be such that
\[
\varphi\left(  x\right)  =\left(  x_{1}+l_{1}\left(  x_{2},\cdots
x_{d}\right)  ,x_{2}+l_{2}\left(  x_{3},\cdots,x_{d}\right)  ,\cdots
,x_{d-1}+l_{d-1}\left(  x_{d}\right)  ,x_{d}\right)
\]
for some $C^{1}$-functions $l_{1},l_{2},\cdots,l_{d-1}$. Then $E({\mathbb{Z}%
}^{d},\varphi)$ is an orthonormal basis for $L^{2}[0,1]^{d}$.
\end{theorem}

\begin{remark}
Let $M$ be a square matrix of order $d$ with integer entries satisfying $|\det
M|=1.$ Since ${\mathbb{Z}}^{d}$ is invariant under the action of the transpose
of $M,$ it is easy to verify that $E({\mathbb{Z}}^{d},M\varphi)$ is an
orthonormal basis for $L^{2}[0,1]^{d}$ if and only if $E({\mathbb{Z}}%
^{d},\varphi)$ is an orthonormal basis for $L^{2}[0,1]^{d}.$ More generally
for arbitrary $A\in GL\left(  d,\mathbb{R}\right)  $, the system
$E({\mathbb{Z}}^{d},A\varphi)$ is an orthonormal basis for $L^{2}[0,1]^{d}$ if
and only if $E(A^{T}{\mathbb{Z}}^{d},\varphi)$ is an orthonormal basis for
$L^{2}[0,1]^{d}.$
\end{remark}

In setting the stage for the proof of Theorem \ref{unipotent}, recall the
following. Let $\Omega$ be a subset of ${\mathbb{R}}^{d}.$ We say that
$\Omega$ is a translational tile by a set ${\mathcal{J}}$ if
\[
m((\Omega+t)\cap(\Omega+{t}^{\prime}))=0\ \forall t\neq t^{\prime}%
\in{\mathcal{J}}\ \mbox{and}\ \bigcup_{t\in{\mathcal{J}}}(\Omega
+t)={\mathbb{R}}^{d}.
\]
When the first condition described above holds, we say that $\Omega$ is a
\textit{packing set}. Given a lattice $\Gamma=A({\mathbb{Z}}^{d})$ where $A$
is an invertible matrix of order $d,$ the quantity $|\det(A)|$ is called the
volume of $\Gamma,$ and it is not difficult to verify that \textit{if $\Omega$
is a packing set by a lattice $\Gamma$ and $m(\Omega)=|\det(A)|$, then
$\Omega$ tiles ${\mathbb{R}}^{d}$ translationally with $\Gamma$} (see e.g.
\cite[Theorem 2.1]{GLW2015}). Furthermore, the following is a well-known
result due to Fuglede \cite{Fug74}

\begin{theorem}
\label{fuglede} Let $\Gamma$ be a full-rank lattice of ${\mathbb{R}}^{d}$.
Then $L^{2}(\Omega)$ admits an exponential orthogonal basis $E(\Gamma)$ if and
only if $\Omega$ is a translational tile by the dual lattice of $\Gamma$ (i.e.
$\Gamma^{\perp}: = \{x\in{\mathbb{R}}^{d}: x\cdot\gamma\in{\mathbb{Z}}
\ \forall\ \gamma\in\Gamma\}$).
\end{theorem}

\medskip

\noindent\textit{Proof of Theorem \ref{unipotent}.} In light of Theorem
\ref{theorem_basis_frame}, to prove Theorem \ref{unipotent}, it suffices to
establish the following: (i) $\varphi$ is $\mu$-essentially injective, (ii)
$\varphi_{\ast}m_{[0,1]^{d}}=m_{\varphi([0,1]^{d})}$ and (iii) $E({\mathbb{Z}%
}^{2})$ is an orthonormal basis for $L^{2}(\varphi([0,1]^{d})$.

For the first part, we only need to show that $\varphi$ is an injective map.
Indeed, assuming that $\varphi(x)=\varphi(y)$, the last coordinates of $x$ and
$y$ are equal to each other. Next, by assumption, $x_{d-1}+l_{d-1}\left(
x_{d}\right)  =y_{d-1}+l_{d-1}\left(  y_{d}\right)  $ and straightforward
calculations show that $x_{d-1}=y_{d-1}$. Proceeding in this fashion, we
establish $x=y,$ showing that $\varphi$ is injective.

To show (ii), we first verify that the Jacobian of $\varphi$ is equal to one.
Indeed, since
\[
\varphi\left(  x\right)  =x+\left(  l_{1}\left(  x_{2},\cdots x_{d}\right)
,l_{2}\left(  x_{3},\cdots,x_{d}\right)  ,\cdots,l_{d-1}\left(  x_{d}\right)
,0\right)
\]
and the Jacobian of $\varphi$ is a unipotent matrix. i.e. $J_{\varphi}\left(
x\right)  =I+N\left(  x\right)  $ for some matrix-valued function $N(x)$ such
that it is strictly upper triangular. Hence, $J_{\varphi}$ is upper-triangular
with all ones on its diagonal and $\det J_{\varphi}=1$. Thus, $\varphi$
defines a local diffeomorphism and is also injective, and as such $\varphi$
induces a $C^{1}$-diffeomorphism between its domain and its range.
Consequently, $\varphi_{\ast}m_{[0,1]^{d}}=m_{\varphi([0,1]^{d})}.$

Finally, to prove (iii), it suffices to show that (see to Theorem
\ref{fuglede}) $\Omega:=\varphi\left(  \lbrack0,1]^{d}\right)  $ is a
translational tile for ${\mathbb{R}}^{d}$.
%In other words, we need to prove that $$ m((\Omega + k)\cap (\Omega+\ell) )= 0 \ \forall k\ne \ell\in{\mathbb Z}^d, \ \mbox{and} \ \bigcup_{k\in{\mathbb Z}^d} (\Omega+k) ={\mathbb R}^d. $$
To this end, we first claim that $\left(  \varphi\left(  \left[  0,1\right)
^{d}\right)  +k\right)  \cap\varphi\left(  \left[  0,1\right)  ^{d}\right)  $
is empty whenever $k$ is a nonzero element of $\mathbb{Z}^{d}.$ Indeed,
suppose that $x,y\in\left[  0,1\right)  ^{d}$ such that $\varphi\left(
x\right)  -\varphi\left(  y\right)  =k\in\mathbb{Z}^{d}.$ This gives the
following system of equations%
\[
\left\{
\begin{array}
[c]{c}%
k_{1}=x_{1}-y_{1}+l_{1}\left(  x_{2},\cdots x_{d}\right)  -l_{1}\left(
y_{2},\cdots,y_{d}\right) \\
k_{2}=x_{2}+l_{2}\left(  x_{3},\cdots x_{d}\right)  -y_{2}-l_{2}\left(
y_{3},\cdots y_{d}\right) \\
\vdots\\
k_{d-1}=x_{d-1}+l_{d-1}\left(  x_{d}\right)  -y_{d-1}-l_{d-1}\left(
y_{d}\right) \\
k_{d}=x_{d}-y_{d}%
\end{array}
\right.  .
\]
Since $x_{d}-y_{d}\in\left(  -1,1\right)  \cap\mathbb{Z}$, it must be the case
that $k_{d}=0.$ This, however, implies that $x_{d}=y_{d}.$ Therefore,
\[
k_{d-1}=x_{d-1}+l_{d-1}\left(  x_{d}\right)  -y_{d-1}-l_{d-1}\left(
x_{d}\right)  =x_{d-1}-y_{d-1}\in\left(  -1,1\right)  \cap\mathbb{Z}%
\]
and $x_{d-1}=y_{d-1}.$ Proceeding inductively, we obtain $x=y$. This implies
that $\varphi\left(  \left[  0,1\right)  ^{d}\right)  $ is a packing set for
${\mathbb{R}}^{d}$ associated with the lattice ${\mathbb{Z}}^{d}$.
Additionally, since $\varphi$ is a measure-preserving map, $m(\varphi\left(
\left[  0,1\right)  ^{d}\right)  )=1$ and this shows that $\varphi\left(
\left[  0,1\right)  ^{d}\right)  )$ tiles ${\mathbb{R}}^{d}$ by ${\mathbb{Z}%
}^{d}.$ \qquad$\Box$

\subsection{Necessary conditions for $d=2$.}

In the subsequent subsection, we will prove that for the special case where
$d=2,$ Theorem \ref{unipotent} is the best result that can be obtained under
the restriction that the Jacobian of $\varphi$ is upper-triangular, $\varphi$
is invertible, $\varphi\left(  \left[  0,1\right)  ^{2}\right)  $ tiles
$\mathbb{R}^{2}$ by $\mathbb{Z}^{2}$ and satisfies some additional technical
restrictions which we shall clarify.

\begin{theorem}
\label{case_2D}Let $\varphi:\mathbb{R}^{2}\rightarrow\mathbb{R}^{2}$ such that
$\varphi\left(  x_{1},x_{2}\right)  =\left(  \varphi_{1}\left(  x_{1}%
,x_{2}\right)  ,\varphi_{2}\left(  x_{1},x_{2}\right)  \right)  $ for some
bivariate, real-valued functions $\varphi_{1}$ and $\varphi_{2}.$ Assuming
additionally that (a) all second-order mixed partial derivatives of
$\varphi_{1},\varphi_{2}$ are continuous (b) $J_{\varphi}\left(  x_{1}%
,x_{2}\right)  $ is an upper-triangular matrix, (c) $\det J_{\varphi}\left(
x_{1},x_{2}\right)  =1,$ then
\[
\varphi\left(  x_{1},x_{2}\right)  =\left(  z\left(  x_{2}\right)
x_{1}+f\left(  x_{2}\right)  ,\int_{1}^{x_{2}}\dfrac{1}{z\left(  t\right)
}dt+K\right)
\]
for some constant $K$ and some functions $z,f\in C^{1}\left(  \mathbb{R}%
\right)  .$ Moreover, if $z$ is taken to be the constant function $1$, then
$\varphi$ is necessarily as described in Theorem \ref{unipotent} and
$E({\mathbb{Z}}^{d},\varphi)$ is an orthonormal basis for $L^{2}\left(
\left[  0,1\right]  ^{2}\right)  .$
\end{theorem}

\begin{remark}
The assumption that $z$ is equal to the constant function $1$ cannot generally
be removed in Proposition \ref{case_2D} without affecting its conclusion. For
instance if $z\left(  x_{2}\right)  =e^{x_{2}}$ and $f\left(  x_{2}\right)
=0$ then%
\[
\varphi\left(  x_{1},x_{2}\right)  =\left(  e^{x_{2}}x_{1},\sinh\left(
x_{2}\right)  -\cosh\left(  x_{2}\right)  +K\right)  .
\]
for some constant $K.$ Next, the Jacobian \ of $\varphi$ has for determinant
\[
\det\left[
\begin{array}
[c]{cc}%
e^{x_{2}} & x_{1}e^{x_{2}}\\
0 & \cosh\left(  x_{2}\right)  -\sinh\left(  x_{2}\right)
\end{array}
\right]  =\left(  \cosh x_{2}\right)  e^{x_{2}}-\left(  \sinh x_{2}\right)
e^{x_{2}}=1.
\]
Therefore, $\varphi$ is Lebesgue-measure preserving. However, it is easy to
verify that the collection $E({\mathbb{Z}}^{d},\varphi)$ is not an orthonormal
basis for $L^{2}\left(  \left[  0,1\right)  ^{2}\right)  $ since the set
$\varphi\left(  \left[  0,1\right)  ^{2}\right)  \cap\left(  \varphi\left(
\left[  0,1\right)  ^{2}\right)  +\left(  1,0\right)  \right)  $ has a
positive Lebesgue measure in $\mathbb{R}^{2}.$
\end{remark}

Let $\varphi:\mathbb{R}^{2}\rightarrow\mathbb{R}^{2}$ such that $\varphi
\left(  x_{1},x_{2}\right)  =\left(  \varphi_{1}\left(  x_{1},x_{2}\right)
,\varphi_{2}\left(  x_{1},x_{2}\right)  \right)  $ for some bivariate,
real-valued functions $\varphi_{1}$ and $\varphi_{2}.$ Assuming additionally
that all partial derivatives of $\varphi_{1}$ and $\varphi_{2}$ are defined,
and
\[
\left\vert \det J\varphi\right\vert =\left\vert \dfrac{\partial\varphi_{1}%
}{\partial x_{1}}\cdot\dfrac{\partial\varphi_{2}}{\partial x_{2}}%
-\dfrac{\partial\varphi_{2}}{\partial x_{1}}\cdot\dfrac{\partial\varphi_{1}%
}{\partial x_{2}}\right\vert =1,
\]
there does not seem to be a simple way to explicitly described all such
functions. However, we will prove that under the additional assumptions that
$\varphi\left(  x_{1},x_{2}\right)  =\left(  \varphi_{1}\left(  x_{1}%
,x_{2}\right)  ,\varphi_{2}\left(  x_{2}\right)  \right)  $ and all second
order mixed partial derivatives are continuous, $\varphi$ can be described
quite explicitly as stated in the lemma below

\begin{lemma}
\label{tiling}Let $\varphi:\mathbb{R}^{2}\rightarrow\mathbb{R}^{2}$ such that
$\varphi\left(  x_{1},x_{2}\right)  =\left(  \varphi_{1}\left(  x_{1}%
,x_{2}\right)  ,\varphi_{2}\left(  x_{2}\right)  \right)  $ for some
bivariate, real-valued functions $\varphi_{1}$ and $\varphi_{2}.$ Assuming
additionally that all second-order mixed partial derivatives of $\varphi
_{1},\varphi_{2}$ are continuous then the following are equivalent.

\begin{enumerate}
\item $J\varphi\left(  x_{1},x_{2}\right)  $ is an upper-triangular matrix and
$\det\left(  J\varphi\left(  x_{1},x_{2}\right)  \right)  =1$ for all $\left(
x_{1},x_{2}\right)  \in\mathbb{R}^{2}.$

\item There exist differentiable functions $z,f\in C^{1}\left(  \mathbb{R}%
\right)  ,z\neq0$ and some constant $K$ such that $\varphi\left(  x_{1}%
,x_{2}\right)  =\left(  z\left(  x_{2}\right)  x_{1}+f\left(  x_{2}\right)
,\int_{1}^{x_{2}}\frac{1}{z\left(  \tau\right)  }d\tau+K\right)  .$
\end{enumerate}
\end{lemma}

\begin{proof}
To prove that (2) implies (1), we verify that the Jacobian of the map
$\varphi\left(  x_{1},x_{2}\right)  =\left(  z\left(  x_{2}\right)
x_{1}+f\left(  x_{2}\right)  ,\int_{1}^{x_{2}}\frac{1}{z\left(  \tau\right)
}d\tau+K\right)  $ is given by
\[
J\varphi\left(  x_{1},x_{2}\right)  =\left[
\begin{array}
[c]{cc}%
z\left(  x_{2}\right)  & f^{\prime}\left(  x_{2}\right)  +x_{1}z^{\prime
}\left(  x_{2}\right) \\
0 & \dfrac{1}{z\left(  x_{2}\right)  }%
\end{array}
\right]  .
\]
For the converse, assume that $J\varphi\left(  x_{1},x_{2}\right)  $ is an
upper-triangular matrix and $\det\left(  J\varphi\left(  x_{1},x_{2}\right)
\right)  =1$ for all $\left(  x_{1},x_{2}\right)  \in\mathbb{R}^{2}.$ In other
words,%
\[
J\varphi\left(  x,y\right)  =\left[
\begin{array}
[c]{cc}%
\dfrac{\partial\varphi_{1}\left(  x_{1},x_{2}\right)  }{\partial x_{1}} &
\dfrac{\partial\varphi_{1}\left(  x_{1},x_{2}\right)  }{\partial x_{2}}\\
\dfrac{\partial\varphi_{2}\left(  x_{1},x_{2}\right)  }{\partial x_{1}} &
\dfrac{\partial\varphi_{2}\left(  x_{1},x_{2}\right)  }{\partial x_{2}}%
\end{array}
\right]  =\left[
\begin{array}
[c]{cc}%
\dfrac{\partial\varphi_{1}\left(  x_{1},x_{2}\right)  }{\partial x_{1}} &
\dfrac{\partial\varphi_{1}\left(  x_{1},x_{2}\right)  }{\partial x_{2}}\\
0 & \left[  \dfrac{\partial\varphi_{1}\left(  x_{1},x_{2}\right)  }{\partial
x_{1}}\right]  ^{-1}%
\end{array}
\right]  .
\]
By Clairaut's theorem, since all second-order mixed partial derivatives of
$\varphi_{1},\varphi_{2}$ are continuous,
\[
0=\frac{\partial}{\partial x_{2}}\left(  \dfrac{\partial\varphi_{2}\left(
x_{1},x_{2}\right)  }{\partial x_{1}}\right)  =\frac{\partial}{\partial x_{1}%
}\left(  \dfrac{\partial\varphi_{2}\left(  x_{1},x_{2}\right)  }{\partial
x_{2}}\right)  =\frac{\partial}{\partial x_{1}}\left(  \left[  \dfrac
{\partial\varphi_{1}\left(  x_{1},x_{2}\right)  }{\partial x_{1}}\right]
^{-1}\right)
\]
and
\[
-\left[  \dfrac{\partial\varphi_{1}\left(  x_{1},x_{2}\right)  }{\partial
x_{1}}\right]  ^{-2}\cdot\left(  \frac{\partial^{2}\varphi_{1}\left(
x_{1},x_{2}\right)  }{\partial x_{1}^{2}}\right)  =0.
\]
As a result, $\frac{\partial^{2}\varphi_{1}\left(  x_{1},x_{2}\right)
}{\partial x_{1}^{2}}=0$ and this holds if and only if $\varphi_{1}\left(
x_{1},x_{2}\right)  =f\left(  x_{2}\right)  +x_{1}z\left(  x_{2}\right)  $ for
some $z,f\in C^{1}\left(  \mathbb{R}\right)  ,z\neq0.$ On the other hand,
\[
\dfrac{\partial\varphi_{2}\left(  x_{1},x_{2}\right)  }{\partial x_{2}%
}=\left[  \dfrac{\partial\varphi_{1}\left(  x_{1},x_{2}\right)  }{\partial
x_{1}}\right]  ^{-1}\Leftrightarrow\varphi_{2}\left(  x_{1},x_{2}\right)
=\int_{1}^{x_{2}}\frac{1}{z\left(  \tau\right)  }d\tau+K\left(  x_{1}\right)
.
\]
Finally, since
\[
\frac{\partial}{\partial x_{1}}\left(  \varphi_{2}\left(  x_{1},x_{2}\right)
\right)  =\frac{\partial}{\partial x_{1}}\left(  \int_{1}^{x_{2}}\frac
{1}{z\left(  \tau\right)  }d\tau+K\left(  x_{1}\right)  \right)  =K^{\prime
}\left(  x_{1}\right)  =0
\]
$K$ must be a constant quantity.
\end{proof}

\medskip

\noindent\textit{Proof of Theorem \ref{case_2D}.} The first part of the
theorem is proved in Lemma \ref{tiling}. Assume that $z\left(  x_{2}\right)
=1$ for all $x_{2}\in\mathbb{R}$. Then
\[
\varphi\left(  x_{1},x_{2}\right)  =\left(  x_{1}+f\left(  x_{2}\right)
,\int_{1}^{x_{2}}dt+K\right)  =\left(  x_{1}+f\left(  x_{2}\right)
,\ x_{2}-1+K\right)
\]
for some constant $K$ and Theorem \ref{case_2D} is a direct consequence of
Theorem \ref{unipotent}. \qquad$\Box$

\medskip

We remark that the converse of Theorem \ref{case_2D} is generally false, as
shown below.

\begin{proposition}
\label{2D}Under the assumption stated in Theorem \ref{case_2D}, if
$z(x_{2})>1$ on some subset of positive measure in $\left(  0,1\right)  $ then
$\varphi\left(  \left[  0,1\right)  ^{2}\right)  $ does not tile
$\mathbb{R}^{2}$ by $\mathbb{Z}^{2}.$
%\textcolor{red}{The case where $z<1$ on a subset of positive Lebesgue measure in $(0,1)$ is not clear to me.}

\end{proposition}

\begin{proof}
Suppose that $z>1$ on some subset of positive measure in $\left(  0,1\right)
.$ Letting $\psi\left(  t,\xi\right)  =z\left(  t\right)  \xi$ be a function
defined on $\left(  0,1\right)  \times\left(  -1,1\right)  ,$ the range of
$\psi$ is given by the set ${\bigcup\limits_{t\in\left(  0,1\right)  }}\left(
z\left(  t\right)  \left(  -1,1\right)  \right)  $ and the set $\left\{
z\left(  t\right)  \left(  -1,1\right)  :t\in\left(  0,1\right)  \right\}
\cap\mathbb{Z}\backslash\left\{  0\right\}  $ is not empty. Next, observe that
for points $\left(  x_{1},x_{2}\right)  ,\left(  y_{1},x_{2}\right)  $
contained in the open set $\left(  0,1\right)  ^{2},$ we have
\[
\varphi\left(  x_{1},x_{2}\right)  =\left(  z\left(  x_{2}\right)
x_{1}+f\left(  x_{2}\right)  ,\int_{1}^{x_{2}}\frac{1}{z\left(  t\right)
}dt+K\right)
\]
and
\[
\varphi\left(  y_{1},x_{2}\right)  =\left(  z\left(  x_{2}\right)
y_{1}+f\left(  x_{2}\right)  ,\int_{1}^{x_{2}}\frac{1}{z\left(  t\right)
}dt+K\right)  .
\]
Taking the difference of the points $\varphi\left(  x_{1},x_{2}\right)  $ and
$\varphi\left(  y_{1},x_{2}\right)  ,$ gives
\[
\varphi\left(  y_{1},x_{2}\right)  -\left(  p\left(  x_{1},x_{2}\right)
\right)  =\left(  z\left(  x_{2}\right)  y_{1}-z\left(  x_{2}\right)
x_{1},0\right)  =\left(  z\left(  x_{2}\right)  \left(  y_{1}-x_{1}\right)
,0\right)  .
\]
Since $y_{1}-x_{1}\in\left(  -1,1\right)  ,$ we may select $x_{2}\in\left(
0,1\right)  $ such that $z\left(  x_{2}\right)  \left(  y_{1}-x_{1}\right)
\in\mathbb{Z}\backslash\left\{  0\right\}  .$ This shows that there exists a
nonzero element $k\in\mathbb{Z}^{2}$ such that the Lebesgue measure of the set
$\left(  \varphi\left(  0,1\right)  ^{2}+k\right)  \cap\varphi\left(  \left(
0,1\right)  ^{2}\right)  $ is strictly positive. This means that
$\varphi\left(  \left[  0,1\right)  ^{2}\right)  $ does not tile
$\mathbb{R}^{2}$ by $\mathbb{Z}^{2}.$
\end{proof}

%In light of the first part of Lemma \ref{2D}, Proposition \ref{case_2D} is
%immediate.

\section{Applications to the discretization problem of representations of
locally compact groups}

In this section, we provide some additional motivation for our work by making
a connection between Question \ref{question} and the discretization problem of
representations of locally compact groups for the construction of frames and
orthogonal bases
\cite{Gr,grochenig2017orthonormal,oussa2018frames,oussa2017regular,oussa2018framesI,MR809337}.

For a large class of (solvable Lie) groups, the explicit realization of the
action of an infinite-dimensional representation is commonly
described in terms of a system involving exponential functions with phases which are generally non-linear. Since these constructions may not be readily accessible in the literature to
non-specialists, we shall present some examples to motivate the results contained in this section and we will connect them with well-studied systems such as wavelets, Gabor wavelets and shearlets. 

\begin{example}
\label{Heisenberg} (Gabor orthonormal bases and the Heisenberg group) Let
$G=\mathbb{R}^{2}\rtimes\mathbb{R}$ be a semi-direct product group with
multiplication given by
\[
\left(  v,t\right)  \left(  w,s\right)  =\left(  w+\left[
\begin{array}
[c]{cc}%
1 & t_{1}\\
0 & 1
\end{array}
\right]  w,t+s\right)  .
\]
$G$ is the three-dimensional Heisenberg group. It is a non-commutative simply
connected nilpotent (solvable) Lie group. Next, let $\pi$ be a function taking
$G$ into the group of unitary operators acting in $L^{2}\left(  \mathbb{R}%
\right)  $ as follows:
\[
\left[  \pi\left(  v,s\right)  f\right]  \left(  t\right)  =e^{2\pi i\left(
1,-t\right)  \cdot\left(  v_{1},v_{2}\right)  }f\left(  t-s\right)  \text{
\ \ \ \ }\left(  f\in L^{2}\left(  \mathbb{R}\right)  \right)  .
\]
It is not hard to verify that $\pi$ is a continuous group homomorphism. In
fact, $\pi$ is an irreducible representation of the Heisenberg group known as
a Schr\"{o}dinger representation. Since $\mathcal{E}=\left\{  \mathcal{E}%
_{k}:x\mapsto e^{2\pi ixk}\chi_{\left[  0,1\right)  }\left(  x\right)
:k\in\mathbb{Z}\right\}  $ is an orthonormal basis for $L^{2}\left[
0,1\right)  $ and $\left\{  \left[  0,1\right) +k:k\in\mathbb{Z}\right\}  $
is a tiling of $\mathbb{R}$, the collection of vectors $\left\{  e^{2\pi
itk}\chi_{\left[  0,1\right)  }\left(  t-l\right)  :k,l\in\mathbb{Z}\right\}
$ forms an orthonormal basis for $L^{2}\left(  \mathbb{R}\right)  .$ Note that
this basis is obtained from a discrete sampling of the orbit of the indicator
function $\chi_{\left[  0,1\right)  }$ under the action of the representation
$\pi.$ In other words, it is possible to discretize $\pi$ to construct an
orthonormal basis for $L^{2}\left(  \mathbb{R}\right)  .$ This is a standard
example that is commonly encountered in time-frequency analysis \cite{Gr}.
\end{example}

The following example suggests that Example \ref{Heisenberg}, as discussed
above, is just a mere occurrence of a much more general phenomenon.

\begin{example}
\label{Example} (Generalized Gabor wavelets) Let $G=\mathbb{R}^{3}\rtimes\mathbb{R}^{2}$ be a semi-direct
product group with multiplication given by
\[
\left(  v,t\right)  \left(  w,s\right)  =\left(  w+\left[
\begin{array}
[c]{ccc}%
1 & t_{1} & \frac{t_{1}^{2}}{2}+t_{2}\\
0 & 1 & t_{1}\\
0 & 0 & 1
\end{array}
\right]  w,t+s\right)  .
\]
Note that although $G$ and $\mathbb{R}^{3}\times\mathbb{R}^{2}$ share the same
topological structure, their group structures are quite different. More
precisely, $G$ is a non-commutative Lie group \cite{Corwin}; and similarly to
the Heisenberg group, its irreducible representations can be exploited to
construct an orthonormal basis for $L^{2}\left(  \mathbb{R}^{2}\right)  $
\cite[Example 31]{oussa2018framesI}. To see this, let $p:$ $\mathbb{R}%
^{2}\rightarrow\mathbb{R}^{3}$ be a vector-valued polynomial map defined as
follows: $p\left(  t_{1},t_{2}\right)  =\left(  1,-t_{1},-t_{2}+\frac
{t_{1}^{2}}{2}\right)  .$ Note that the third coordinate of $p$ is a bivariate
non-linear polynomial. Next, the function $\pi$ mapping the group $G$ into the
group of unitary operators acting in $L^{2}\left(  \mathbb{R}^{2}\right)  $ as
follows%
\[
\left[  \pi\left(  v,s\right)  f\right]  \left(  t\right)  =e^{2\pi ip\left(
t_{1},t_{2}\right)  \cdot v}f\left(  t-s\right)  \text{ \ \ \ }\left(  f\in
L^{2}\left(  \mathbb{R}^{2}\right)  \right)
\]
can be shown to be an irreducible representation (a continuous homomorphism)
of the group $G$. Moreover, it is easy to verify that the system of
exponentials with non-linear phase
\[
\left\{  \left(  t_{1},t_{2}\right)  \mapsto e^{2\pi ip\left(  t_{1}%
,t_{2}\right)  \cdot\left(  0,k_{1},k_{2}\right)  }\chi_{\left[  0,1\right)
^{2}}\left(  t_{1},t_{2}\right)  :\left(  k_{1},k_{2}\right)  \in
\mathbb{Z}^{2}\right\}
\]
is an orthonormal basis for $L^{2}\left(  \left[  0,1\right)  ^{2}\right)  .$
This observation together with the fact that $\left\{  \left[  0,1\right)
^{2}+k:\left(  k_{1},k_{2}\right)  \in\mathbb{Z}^{2}\right\}  $ is a
measurable partition of $\mathbb{R}^{2}$ imply that the collection of vectors
\[
\left\{  \left(  t_{1},t_{2}\right)  \mapsto e^{2\pi i\left(  -\left(
t_{1}+\ell_{1}\right)  ,-\left(  t_{2}+\ell_{2}\right)  +\frac{\left(
t_{1}+\ell_{1}\right)  ^{2}}{2}\right)  \cdot\left(  k_{1},k_{2}\right)  }%
\chi_{\left[  0,1\right)  ^{2}}\left(  t_{1}+\ell_{1},t_{2}+\ell_{2}\right)
:\left(  k_{1},k_{2},\ell_{1},\ell_{2}\right)  \in\mathbb{Z}^{4}\right\}
\]
is an orthonormal basis for $L^{2}\left(  \mathbb{R}^{2}\right)  .$ In other
words, a suitable discretization of the representation $\pi$ gives an
orthonormal basis for $L^{2}\left(  \mathbb{R}^{2}\right)  .$
\end{example}

\begin{example}
(The ax+b group) Let $G=\mathbb{R}\rtimes\mathbb{R}$ be a semidirect product
group equipped with the group operation $\left(  x,t\right)  \left(
y,s\right)  =\left(  y+e^{t}y,t+s\right) .$ Then $G$ is isomorphic to the ax+b
Lie group which is known to be the group theoretical foundation of wavelet
theory \cite{Ole}. Given a fixed positive real number $\ell$ the unitary representation
$\pi_{\ell}$ of $G$ acting in $L^{2}\left(  \mathbb{R}\right) $ as follows:%
\[
\pi_{\ell}\left(  x,t\right)  f\left(  s\right)  =e^{2\pi ie^{-s}\ell x}f\left(
s-t\right)  ,\text{ }\left(  f\in L^{2}\left(  \mathbb{R}\right)  \right)
\]
is irreducible. Put $\varphi\left(  s\right)  =e^{-s}\ell$ and let $L$ be the
Lebesgue measure on the real line. Since $\varphi$ is injective, and since the
pushforward of the Lebesgue measure via $\varphi$ is a weighted Lebesgue
measure of the form $\varphi_{\ast}L=\frac{1}{x}dx$ on $\left(  0,\infty
\right) ,$ we obtain the following. Given any positive real number $\epsilon$
and for a fixed countable set $\Lambda\subset\mathbb{R}$ such that the lower
Beurling density of $\Lambda$ is positive, and $\{e^{2\pi i \lambda\cdot x}\}_{\lambda\in\Lambda}$ forms a Fourier
frame for $L^{2}(\left[  e^{-\epsilon},e^{\epsilon}\right)),$ the  system $\left\{
\pi_{\ell}\left(  \lambda,0\right)  1_{\left[  -\epsilon,\epsilon\right)
}:\lambda\in\Lambda\right\} $ is a Fourier frame for $L^{2}\left(  \left[
-\epsilon,\epsilon\right)  \right)$ (see Theorem \ref{Fourier_Frame}.) Next, since $\left\{  \left[
-\epsilon,\epsilon\right)  +\kappa:\kappa\in2\epsilon\mathbb{Z}\right\} $
tiles the real line, it follows that
\[
\left\{  \pi_{\ell}\left(  e^{\kappa}\lambda,\kappa\right)  1_{\left[
-\epsilon,\epsilon\right)  }:\lambda\in\Lambda\text{ and }\kappa\in
2\epsilon\mathbb{Z}\right\}
\]
is a frame for $L^{2}\left(  \mathbb{R}\right)  .$
\end{example}

\begin{example}
\label{Example} (A shearlet group, \cite{oussa2018frames,oussa2018framesI}) Let $G=\mathbb{R}^{2}\rtimes
\mathbb{R}^{2}$ be a semi-direct product group with multiplication given by
\[
\left(  v,t\right)  \left(  w,s\right)  =\left(  w+\left[
\begin{array}
[c]{cc}%
e^{t_{1}} & t_{2}e^{t_{1}}\\
0 & e^{t_{1}}%
\end{array}
\right]  w,t+s\right)  .
\]
Let $\varphi:$ $\mathbb{R}^{2}\rightarrow\varphi\left(  \mathbb{R}^{2}\right)
$ be a vector-valued smooth map defined as follows: $\varphi\left(
t_{1},t_{2}\right)  =\left(  e^{-t_{1}},-t_{2}e^{-t_{1}}\right)  .$ Define a
unitary representation $\pi$ of $G$ acting in $L^{2}\left(  \mathbb{R}%
^{2}\right)  $ as follows:%
\[
\left[  \pi\left(  v,s\right)  f\right]  \left(  t\right)  =e^{2\pi
i\varphi\left(  t_{1},t_{2}\right)  \cdot v}f\left(  t-s\right)  \text{
\ \ \ }\left(  f\in L^{2}\left(  \mathbb{R}^{2}\right)  \right)  .
\]
Then it is not difficult to verify that $\pi$ is an irreducible representation
of $G$. Note also that $\varphi$ is injective. In fact, $\varphi$ defines a
diffeomorphism between its domain and its range. Next, let $\Lambda$ be a
countable subset of $\mathbb{R}^{2}$ such that the pushforward of Lebesgue
measure on $\left[  0,1\right)  ^{2}$ to $\varphi\left(  \left[  0,1\right)
^{2}\right)  $ is a frame-spectral measure with frame spectrum $\Lambda.$ Then
the system
\[
\left\{  t\mapsto e^{2\pi i\varphi\left(  t+l\right)  \cdot k}\chi_{\left[
0,1\right)  ^{2}}\left(  t+l\right)  :\left(  k,l\right)  \in\Lambda
\times\mathbb{Z}^{2}\right\}
\]
is a frame for $L^{2}\left(  \mathbb{R}^{2}\right)  .$ 
\end{example}

To generalize the examples above, we will appeal to the results in Theorem
\ref{theorem_basis_frame} to derive some sufficient conditions under which a
class of unitary representations of some connected Lie groups
\cite{grochenig2017orthonormal,oussa2018frames,oussa2018framesI} can be
discretized, for the construction of orthogonal bases and frames in
$L^{2}\left(  \mathbb{R}^{m}\right)  $.

To this end, let $G=\mathbb{R}%
^{d}\rtimes\mathbb{R}^{m}$ be a connected semidirect product group endowed
with the following binary operation: $\left(  x,t\right)  \left(  y,s\right)
=\left(  x+t\bullet y,t+s\right)  $, $\left(  x,t\right)  ,\left(  y,s\right)
\in G$ where
\[
t\bullet y=\exp\left(  \sum_{k=1}^{m}t_{k}A_{k}\right)  y
\]
and $A_{1},\cdots,A_{m}$ is a sequence of pairwise commuting square matrices
of order $d$. 

Let $\pi$ be a representation of $G$ acting unitarily in $L^{2}\left(
\mathbb{R}^{m}\right)  $ as follows. Given $f\in L^{2}\left(  \mathbb{R}%
^{m}\right)  ,$ $\left(  x,t\right)  \in G$ and a fixed vector $\ell
\in\mathbb{R}^{d},$ we define%
\[
\pi\left(  x,t\right)  f\left(  s\right)  =e^{2\pi i (\varphi(s)\cdot x) }f\left(  s-t\right)
\]
where the mapping $\varphi:\mathbb{R}^{m}\longrightarrow{\mathbb{R}}^{d}$
given by $\varphi\left(  t\right)  =\exp\left(  -\sum_{k=1}^{m}t_{k}%
A_{k}\right)  ^{T}\ell$ is a smooth function. Note that for each
$x\in{\mathbb{R}}^{d},$ the operator $\pi(x,e)$ acts by multiplication with an
exponential function with a nonlinear phase (generally) and the conormal part
of $G$ acts by translations. In light of these observations, the action of $G$
in $L^{2}\left(  \mathbb{R}^{m}\right)  $ can be viewed as a form of
generalized time-frequency shift. \newline

A straightforward application of Theorem \ref{theorem_basis_frame}, gives the following.

\begin{proposition}
\label{ONB_Lie_Group} If there exist a countable set $\Lambda\subset{\mathbb{R}}^{d}$ and
$\Omega,\Gamma\subset\mathbb{R}^{m}$ such that (a) $\varphi$ is $\mu$-essentially
injective, (b) the pushforward of the Lebesgue measure on $\mathbb{R}^{m}$
restricted to $\Omega$ is a (frame) spectral measure with (frame) spectrum
$\Lambda,$ (c) $\left\{  \Omega+\gamma:\gamma\in\Gamma\right\}$ tiles $\R^m,$
then the system
\[
\left\{  s\mapsto e^{2\pi i\varphi\left(  s-\gamma\right)  \cdot\lambda}1_{\Omega
}\left(  s-\gamma\right)  :\lambda\in\Lambda,\gamma\in\Gamma\right\}
\]
is an orthogonal basis (or a frame) for 
$L^{2}\left(\R^m\right)  .$
In other words,
\[
\left\{  \pi\left(  \gamma\cdot\lambda,\gamma\right)  1_{\Omega}:\left(
\lambda,\gamma\right)  \in\Lambda\times\Gamma\right\}
\]
is an orthogonal basis (or a frame) for $L^{2}\left(  \mathbb{R}^{m}\right)
.$
\end{proposition}

\begin{proof}
By assumption, the system $\left\{  s\mapsto e^{2\pi
i\varphi\left(  s\right)  \cdot\lambda}1_{\Omega}\left(  s\right)  :\lambda
\in\Lambda\right\}  $ is an orthogonal basis (or a frame) for $L^{2}\left(  \Omega\right)
.$ Moreover, since each $\pi\left(  \gamma\right)  ,\gamma\in\Gamma$ is a
unitary operator and since the image of an orthogonal basis (or a frame) under
a unitary map is an orthogonal basis (or a frame), it follows that for a fixed
$\gamma\in\Gamma,$
\[
\pi\left(  0,\gamma\right)  \left(  \left\{  s\mapsto e^{2\pi i\varphi\left(
s\right)  \cdot\lambda}1_{\Omega}\left(  s\right)  :\lambda\in\Lambda\right\}
\right)
\]
is an orthogonal basis (or a frame) for $L^{2}\left(  \Omega+\gamma\right)
.$ Finally, since $\left\{  \Omega+\gamma:\gamma\in\Gamma\right\}  $ is a measurable
partition of $\mathbb{R}^{m},$ it follows that
\begin{align*}
&
%TCIMACRO{\dbigcup \limits_{\gamma\in\Gamma}}%
%BeginExpansion
{\displaystyle\bigcup\limits_{\gamma\in\Gamma}}
%EndExpansion
\pi\left(  0,\gamma\right)  \left(  \left\{  s\mapsto e^{2\pi i\varphi\left(
s\right)  \cdot\lambda}1_{\Omega}\left(  s\right)  :\lambda\in\Lambda\right\}
\right)  \\
& =\left\{  s\mapsto e^{2\pi i\varphi\left(  s-\gamma\right)  \cdot\lambda}1_{\Omega
}\left(  s-\gamma\right)  :\lambda\in\Lambda,\gamma\in\Gamma\right\}
\end{align*}
is an orthogonal basis (or a frame) for $L^{2}\left(  \mathbb{R}%
^{m}\right)  .$
\end{proof}

\section{Open problems}

This paper provides a systematic study about the generalized exponential
system $E(\Lambda,\varphi)$ forming a frame and basis in some $L^{2}(\mu)$. We
are left with many questions that we have not been able to provide a complete answer.

\begin{enumerate}
\item Given any finite Borel measure $\mu$ on ${\mathbb{R}}^{d}$. Due to the
flexibility of Borel measurable functions $\varphi$, is it true that every
$L^{2}(\mu)$ can admit some $E(\Lambda,\varphi)$ as an orthogonal basis? Can
we find a Borel measure that does not admit any orthogonal basis of non-linear phase?

\medskip

\item The middle-third Cantor measures example given in Section \ref{Fractal}
admits orthogonal basis of exponentials with a non-linear phase function
$\varphi$. They are not differentiable everywhere and is drastically different
from our familiar phase functions $\varphi(x) = x$. A natural question here is
that can we find $\varphi$, defined in an open set containing the support, of
better regularity (e.g., $\varphi$ is $C^{\infty}$) so that $E(\Lambda
,\varphi)$ forms an orthogonal basis? For other open, connected sets, do we
have an exponential orthogonal basis with a non-linear phase?

\medskip

\item Theorem \ref{p1} provides a characterization on ${\mathbb{R}}^{1}$ that
a continuous function can be an orthogonal basis for $L^{2}[0,1]$ with integer
frequencies, provided that $\varphi$ preserves measure zero sets. It looks
like that there may exist a continuous function, which will be highly
irregular, such that $E({\mathbb{Z}},\varphi)$ forms a basis or a frame for
$L^{2}[0,1]$. Will there be any such function? Or can we remove the preserving
measure-zero set assumption in Theorem \ref{p1}?

\medskip

\item Theorem \ref{unipotent} provides a large class of functions that
$E({\mathbb{Z}}^{d},\varphi)$ can form an orthogonal basis for $L^{2}%
[0,1]^{d}$. Are there any other $C^{1}$-functions $\varphi$ not of the form
\[
x\mapsto M\left(  x_{1}+l_{1}\left(  x_{2},\cdots x_{d}\right)  ,x_{2}%
+l_{2}\left(  x_{3},\cdots,x_{d}\right)  ,\cdots,x_{d-1}+l_{d-1}\left(
x_{d}\right)  ,x_{d}\right)
\]
for some $C^{1}$-functions $l_{1},l_{2},\cdots,l_{d-1}$ where $M$ is a matrix
in integer entries satisfying $|\det M|=1\ $for which $E({\mathbb{Z}}%
^{d},\varphi)$ is an orthonormal basis for $L^{2}[0,1]^{d}$?
\end{enumerate}

\bibliographystyle{plain}
\bibliography{frtr}

\begin{comment}

\end{comment}

\end{document}